\documentclass{article}

\usepackage{amsthm, amssymb, amsmath}
\usepackage{verbatim}
\usepackage{graphics}      % standard graphics specifications
\usepackage{graphicx}      % alternative graphics specifications
\usepackage{setspace}

\newtheorem{thm}{Theorem}
\newtheorem{prop}[thm]{Proposition}
\newtheorem{cor}[thm]{Corollary}
\newtheorem{lem}[thm]{Lemma}

\newtheorem{defn}{Definition}

\newcommand{\Q}{\mathbb{Q}}
\newcommand{\C}{\mathbb{C}}
\newcommand{\R}{\mathbb{R}}
\newcommand{\Gal}{\textrm{Gal}}
\newcommand{\Z}{\mathbb{Z}}
\newcommand{\mf}{\mathfrak}
\newcommand{\ord}{\textrm{ord}}

\title{An Asymptotic for the Number of Solutions to Linear Equations in Prime Numbers from Specified Chebotarev Classes}
\author{Daniel M. Kane}

\begin{document}
\maketitle

\begin{abstract}
We extend results relating to Vinogradov's three primes Theorem to provide asymptotic estimates for the number of solutions to a given linear equation in three or more prime numbers under the additional constraint that each of the primes involved satisfies specialized Chebotarev conditions.  In particular, we show that such solutions can be expected to exist unless a solution would violate some local constraint.
\end{abstract}

\section{Introduction and Statement of Results}

In 1937 Vinogradov proved that any sufficiently large odd number could be written as the sum of three primes.  In addition, he managed to provide an asymptotic for the number of ways to do so, proving (as stated in Iwaniec-Kowalski (\cite{IK}) Theorem 19.2)

\begin{thm}[Vinogradov]\label{IKThm}
For $N$ a positive integer and $A$ any real number then
\begin{equation}\label{IKThmEqu}
\sum_{n_1+n_2+n_3=N} \Lambda(n_1)\Lambda(n_2)\Lambda(n_3) =  \mathfrak{G}_3(N) N^2 + O(N^2 \log^{-A}(N)),
\end{equation}
where
$$
\mathfrak{G}_3(N)=\frac{1}{2}\prod_{p|N} (1-(p-1)^{-2}) \prod_{p\nmid N} (1+(p-1)^{-3}),
$$
$\Lambda(n)$ is the Von Mangoldt function, and the asymptotic constant in the $O$ depends on $A$.
\end{thm}

It is easy to see that the contribution to the left hand side of Equation \ref{IKThmEqu} coming from one of the $n_i$ a power of prime is negligible, and thus this side of the equation may be replaced by a sum over triples $p_1,p_2,p_3$ of primes that sum to $N$ of $\log(p_1)\log(p_2)\log(p_3)$.  This implies that any sufficiently large odd number can be written as a sum of three primes since $\mathfrak{G}_3(N)$ is bounded below by a constant for $N$ odd.

It should also be noted that the main term, $N^2\mathfrak{G}_3(N)$ can be written as
$$
C_\infty \prod_p C_p,
$$
where
$$
C_\infty = \frac{N^2}{2}
$$
and
$$
C_p = \begin{cases} (1-(p-1)^{-2}) & \textrm{ if } p|N \\ (1+(p-1)^{-3}) & \textrm{ else} \end{cases}.
$$

When written this way, there is a reasonable heuristic explanation for Theorem \ref{IKThm}.  To begin with, the Prime Number Theorem says that the Von Mangoldt function, is approximated by the distribution assigning 1 to each positive integer. The term $C_\infty$ provides an approximation to the number of solutions based on this heuristic.  The $C_p$ can be thought of as corrections to this heuristic. They can be thought of as local contributions coming from congruential information about the primes $p_i$. In particular, $C_p$ can easily be seen to be equal to
$$
\frac{p\#\left\{(n_1,n_2,n_3)\in \left((\Z/p\Z)^*\right)^3: n_1+n_2+n_3 \equiv N \pmod{p} \right\}}{(p-1)^3}.
$$
This can be thought of as a correction factor coming from the fact that no prime (except for $p$) is a multiple of $p$. In particular, $C_p$ is equal to the ratio of the probability that three randomly chosen elements of $(\Z/p\Z)^*$ sum to $N$ modulo $p$ to the probability that three randomly chosen elements of $(\Z/p\Z)^+$ sum to $N$ modulo $p$.

A number of generalizations of Vinogradov's Theorem have since been proven.  Some, such as Zhan in \cite{SmallPrime} and \cite{ClosePrimeSum}, deal with restrictions on the relative sizes of the primes involved.  In particular, \cite{SmallPrime} shows that one of the primes can be taken to be as small as $N^{7/120+\epsilon}$.  It is shown in \cite{ClosePrimeSum} that the primes involved can all be taken to be relatively close to each other.

The problem was generalized to number fields by Tuljaganova in \cite{PrimesInFields2} and later by Noda in \cite{PrimesInFields1}, who ask which elements of the ring of integers can be written as a sum of generators of principle prime ideals.

Several more papers deal with problems where the primes involved are required to be taken from specified subsets of the set of all prime numbers.  In \cite{DensityGoldbach} Li and Pan show that for any three sets of prime numbers with sufficient density that any sufficiently large odd $N$ can be written as a sum of primes, one from each set.  In particular, they show that for any three sets of primes with relative densities $\delta_1,\delta_2$ and $\delta_3$ within the set of all primes so that $\delta_1+\delta_2+\delta_3>2$ that any sufficiently large odd number can be written as a sum of one prime from each set. It is not hard to see that the bound of 2 above is necessary.

A number of authors including Zulauf (\cite{PrimesInProgressionsSums0}), Liu and Zhan (\cite{PrimesInProgressionsOneHalf}), Halupczok (\cite{PrimesInProgressionsSums}) and Meng (\cite{PrimesInProgressionsSums2}) deal with the case where the primes involved are required to be taken from specified arithmetic progressions.

In this paper, we prove a new generalization of Theorem \ref{IKThm}, counting solutions to similar equations where in addition the primes $p_i$ are required to lie in specified Chebotarev classes.  In particular, after fixing Galois extensions $K_i/\Q$ and conjugacy classes $C_i$ of $\Gal(K_i/\Q)$, we find an asymptotic for the sum of $\prod_{i=1}^k \log(p_i)$ over primes $p_1,\ldots,p_k\leq X$ so that $[K_i/\Q, p_i] = C_i$ for each $i$ and $\sum_{i=1}^k a_i p_i =N$.  Note that the results on writing $N$ as a sum of primes from arithmetic progressions, will follow as a special case of this when $K_i$ is abelian over $\Q$ (although our bounds are probably worse).  In particular we prove:

\begin{thm}\label{mainThm}
Let $k\geq 3$ be an integer.  Let $K_i/\Q$ be finite Galois extensions $(1\leq i \leq k)$ and $G_i=\Gal(K_i/\Q)$.  Let $a_1,\ldots,a_k$ be non-zero integers with no common divisor.  Let $C_i$ be a conjugacy class of $G_i$ for each $i$.  Let $K_i^a$ be the maximal abelian extension of $\Q$ contained in $K_i$, and let $D_i$ be its discriminant.  Let $D$ be the least common multiple of the $D_i$.  Let $H_i^0$ be the subgroup of $(\Z/D\Z)^*$ corresponding to $K_i^a$ via global class field theory.  Let $H_i$ be the coset of $H_i^0$ corresponding to the projection of an element of $C_i$ to $\Gal(K_i^a/\Q)$.  Additionally let $N$ be an integer and let $A$ and $X$ be positive numbers, then
\begin{equation}\label{asymptEqu}
\sum_{\substack{p_i \leq X \\ [K_i/\Q, p_i] = C_i \\ \sum_i a_i p_i = N}} \prod_{i=1}^k \log(p_i) =
\left(\prod_{i=1}^k \frac{|C_i|}{|G_i|} \right)C_\infty C_{D} \left(\prod_{p \nmid D} C_p\right) + O\left(X^{k-1} \log^{-A}(X)\right),
\end{equation}
where the sum of the right hand side is over sets of prime numbers $p_1,\ldots,p_k\leq X$ so that $\sum_{i=1}^k a_i p_i = N$, and so that the Artin symbol $[K_i/\Q, p_i]$ lands in the conjugacy class $C_i$ of $G_i$ for all $1\leq i\leq k$.  On the right hand side,
$$
C_\infty = \frac{1}{\sum_{i=1}^k a_i^2}\int_{\substack{x_i\in [0,X] \\ \sum_i a_i x_i = N}} \left(\sum_{i=1}^k a_i \frac{\partial}{\partial x_i} \right)dx_1 \wedge dx_2 \wedge \ldots \wedge dx_k,
$$
$$
C_D = D\left(\frac{\#\{(x_i)\in \left((\Z/D\Z)^*\right)^k : x_i \in H_i, \sum_{i=1}^k a_i x_i \equiv N \pmod D \}}{\prod_{i=1}^k |H_i|}\right),
$$
and the second product is over primes $p$ not dividing $D$ of
$$
C_p = p\left( \frac{\#\{(x_i) \in \left((\Z/p\Z)^*\right)^k:  \sum_{i=1}^k a_i x_i \equiv N \pmod p \}}{(p-1)^k}\right).
$$
The implied constant in the $O$ term may depend on $k,K_i,C_i,a_i,$ and $A$, but not on $X$ or $N$.  Additionally, if $k=2$ and $K_i,C_i,a_i,A,X$ are fixed, then Equation (\ref{asymptEqu}) holds for all but $O(X\log^{-A}(X))$ values of $N$.
\end{thm}

The introduction of Chebotarev classes leads to two main differences between our asymptotic and the classical one.  For one, the Chebotarev Density Theorem tells us that there are fewer primes in these Chebotarev classes than out of them and causes us to introduce a factor of $\prod_{i=1}^k \left( \frac{|C_i|}{|G_i|}\right)$.  Secondly, Global Class Field Theory tells us that the prime $p_i$ will necessarily lie in the subset $H_i$ of $(\Z/D\Z)^*$, giving us the correction factor $C_D$ rather than $\prod_{p|D}C_p$ to account for required congruence relations that these primes satisfy.

It should be noted that the error term is $o(X^{k-1} \log(X)^{-1})$, whereas if $N$ is bounded away from both the largest and smallest possible values that can be taken by $\sum_{i} a_ix_i$ for $x_i\in [0,X]$, then $C_\infty$ will be on the order of $X^{k-1}$.  For $K_i,C_i$ fixed, the first term on the right hand side is a constant. Although it depends on $N$, $C_D$ will be bounded away from both $0$ and $\infty$ unless $N$ cannot be written as a sum $\sum a_i x_i$ with $x_i\in H_i$. Lastly, for $p\nmid Dn\prod_i a_i$, inclusion-exclusion tells us that $C_p = 1+O(p^{-2})$, and for $p|N$, $p\nmid D\prod_i a_i$, $C_p = 1+O(p^{-1})$.  This means that unless $C_p=0$ for some $p$, $\prod_p C_p$ is within a bounded multiple of $\prod_{p|N}(1+O(p^{-1})) = \exp(O(\log\log\log N)).$  Therefore, unless $C_D=0$, $C_p=0$ for some $p$, or $N$ is near the boundary of the available range, the main term on the right hand side of Equation \eqref{asymptEqu} dominates the error.

\section{Overview}

Our proof will closely mimic the proof in \cite{IK} of Theorem \ref{IKThm}.  We provide a brief overview of the proof given in \cite{IK}, discuss our generalization and provide an outline for the rest of the paper.

\subsection{The Proof of Theorem \ref{IKThm}}

On a very general level, the proof given in \cite{IK} depends on writing
$$
\Lambda = \Lambda^\sharp + \Lambda^\flat.
$$
Here $\Lambda^\sharp$ is a nice approximation to the Von Mangoldt function obtained essentially by sieving out multiples of small primes and $\Lambda^\flat$ is an error term.  It is relatively easy to deal with the sum
$$
\sum_{n_1+n_2+n_3=N}\Lambda^\sharp(n_1)\Lambda^\sharp(n_2)\Lambda^\sharp(n_3),
$$
yielding the main term in Equation \eqref{IKThmEqu}.  This leaves additional terms, each involving at least one $\Lambda^\flat$.  These terms are dealt with by showing that $\Lambda^\flat$ is small in the sense that its generating function has small $L^\infty$ norm.

To prove this bound on $\Lambda^\flat$, Iwaniec and Kowalski make use of Theorem 13.10 of \cite{IK}, which states that for any $A$
$$
\sum_{m\leq x} \mu(m)e^{2\pi i \alpha m} \ll x \log^{-A}(x)
$$
($\mu$ is the M\"{o}bius function) with the implied constant depending only on $A$.  This in turn is proved by considering separately the case where $\alpha$ is near a rational number of small denominator and the case where it is not.

If $\alpha$ is close to a rational number, the sum can be bounded through the use of Dirichlet $L$-functions. In particular, one has bounds on $\sum_{n\leq x} \chi(n)\mu(n)$ for $\chi$ a Dirichlet character (\cite{IK} (5.80)). To prove this, Iwaniec and Kowalski use both Theorem 5.13 of \cite{IK}, which gives bounds on the sums of coefficients of the logarithmic derivative of an $L$-function, and some bounds on zero-free regions and Siegel zeroes.

If $\alpha$ is not well approximated by a rational number with small denominator, an appropriate bound is proved by rewriting the sum using some combinatorial identities (\cite{IK} (13.39)) and using the quadratic form trick.  (The actual bound obtained is given in \cite{IK} Theorem 13.9.)

\subsection{Outline of Our Proof}

Our proof of Theorem \ref{mainThm} is similar in spirit to the proof of Theorem \ref{IKThm} given in \cite{IK}.  We differ in a few ways, some just in the way we choose to organize our information and some from necessary complications due to the increased generality.  We provide below an outline of our proof and a comparison of our techniques to those used in \cite{IK}.

Instead of dealing directly with $\Lambda,\Lambda^\sharp$ and $\Lambda^\flat$ as is done in \cite{IK}, we instead deal directly with their generating functions.  In Section \ref{FGdefSec}, we define $G$, which is our equivalent of the generating function for $\Lambda$.  As it turns out, $G$ is somewhat difficult to deal with directly, so we define a related function $F$, that is better suited for techniques involving Hecke $L$-functions.  In Proposition \ref{GFRelationProp} we prove that we can write $G$ approximately as an appropriate sum of $F$'s.

In Section \ref{LocDefSec} we define $G^\sharp$ and $G^\flat$, which are analogues of the generating functions for $\Lambda^\sharp$ and $\Lambda^\flat$.  We also define analogous $F^\sharp$ and $F^\flat$. The sieving technique that we use to write $G^\sharp$ is not quite analogous to that used in \cite{IK}. Essentially, we write our version of $\Lambda^\sharp$ as a product of local factors. This will produce some sums over smooth numbers later in our analysis, so in Lemma \ref{SmoothNumberLem} we bound the number of smooth numbers, so that we may bound errors coming from sums over them.

We next work on proving that $F^\flat$ has small $L^\infty$ norm (this is somewhat equivalent to \cite{IK} showing that generating functions of $\Lambda^\flat$ or $\mu$ are small).  As in \cite{IK}, we split into two cases based on whether or not we are near a rational number.

In Section \ref{alphaSmoothSec}, we deal with the approximation near rationals.  First, in Section \ref{LFunctSec}, we generalize some necessary results about $L$-functions and Siegel zeroes.  In Section \ref{FSmoothApproxSec}, we use these to produce an approximation of $F$, and in Section \ref{FSharpSmoothApproxSec}, we show that this also approximates $F^\sharp$.

In Section \ref{alphaRoughSec}, we deal with showing that $F^\flat$ is small away from rationals.  It should be noted that while the rest of this paper generalizes the corresponding proof in \cite{IK} in a relatively straightforward way by use of standard results, something new is needed for this Section.  The primary reason for this is that while Vinogradov's bound on exponential sums over prime numbers reduces the sum in question to exponential sums over arithmetic progressions, the analogous argument in our case requires bounding sums of the form $\sum e^{2\pi i \alpha N(\mathfrak{a})}$ over ideals $\mathfrak{a}$ in a number field.  To deal with this issue, we will make use of results about exponential sums of polynomials.  Unfortunately, standard results of this type will not be strong enough when the leading term of the polynomial is approximated by a rational number with relatively small (polylogarithmic) denominator. Thus, we require a new bound of this type which is given by Lemma \ref{polySumBoundLem} below.  In the process of deriving this Lemma, we need some results about when multiples of a number with poor rational approximation have a good rational approximation, which we prove in Section \ref{RatApproxSec}.  In Section \ref{ExpSumsSec}, we use this result to prove bounds on sums of the type described above, and in Section \ref{FRoughBoundSec} use these to obtain the necessary control on $F$. In Section \ref{FSharpRoughSec}, we prove bounds for $F^\sharp$, and thus on $F^\flat$.

In Section \ref{GSec}, we use our bounds on $F^\flat$ to prove bounds on $G^\flat$.  Finally, in Section \ref{proofSec}, we use this bound to prove Theorem \ref{mainThm}.  In Section \ref{ErrorTermSec}, we introduce the appropriate product generating functions, and deal with the terms coming from $G^\flat$'s.  In Section \ref{mainTermSec}, we produce the main term of our Theorem.

Finally, in Section \ref{appSec}, we show an application of our Theorem to constructing elliptic curves whose discriminants split completely over specified number fields.

\section{Preliminaries}\label{prelimSec}

In this Section, we introduce some of the basic terminology and results that will be used throughout the rest of the paper. In Section \ref{NotSec}, we briefly recall some asymptotic notation. In Section \ref{CFTSec}, we recall some of the basic facts from class field theory that will be used later. In Section \ref{FGdefSec}, we define the functions $F$ and $G$ along with some of the basic facts relating them.  In Section \ref{LocDefSec}, we define $F^\sharp$ and $G^\sharp$ along with some related terminology and again prove some basic facts.  Finally, in Section \ref{SmoothSec}, we prove a result on the distribution of smooth numbers that will prove useful to us later.

\subsection{Asymptotic Notation}\label{NotSec}

Throughout we use $O(X)$ to denote a quantity whose absolute value is bounded above by some constant times $X$. Let $\Omega(X)$ denote a positive quantity that is bounded below by some constant times the absolute value of $X$. We use, $\Theta(X)$ will be used to denote a quantity which is both $O(X)$ and $\Omega(X)$. Throughout the paper the implied constants will potentially depend on the number fields $K_i,K,L,$ etc. in question, but upon nothing else unless otherwise stated.

\subsection{Class Field Theory}\label{CFTSec}

Specifying the Artin symbol of a prime will sometimes force congruence conditions on it coming from global class field theory.  In this Section, we review some of the basic facts of this theory that will be needed later. A reader interested in proofs of these results is encouraged to read Milne \cite{milne}. The input that we require from class field theory can be summarized in the following theorem:

\begin{thm}\label{CFTTheorem}
There is a one-to-one correspondence between Galois extensions $K/\Q$ with abelian Galois group and pairs $(H,N)$ where $N$ is a positive integer and $H$ is a subgroup of $(\Z/N\Z)^*$ so that $H$ is not periodic modulo $M$ for any $M$ strictly dividing $N$.  Furthermore, if $K$ is the extension corresponding to this pair $(H,N)$, then $K$ is ramified exactly at the primes dividing $N$ (in fact $N$ divides the discriminant of $K$) and there exists an isomorphism $\varphi:(\Z/N\Z)^*/H \rightarrow \Gal(K/\Q)$ so that for any rational prime $p$ not dividing $N$,
$$
[K/\Q,p] = \varphi(p\pmod{N}).
$$
In particular, if $\chi:\Gal(K/\Q)\rightarrow \C^*$ is a character, then $\chi([K/\Q,p]) = \psi(p)$ for some Dirichlet character $\psi$.

More generally, if $K/L$ is any abelian extension of number fields, and $\chi:\Gal(K/L)\rightarrow \C^*$ is a character, then there exists a Grossencharacter $\psi$ so that for primes $\mathfrak{p}$ relatively prime to the discriminant of $K$, we have $\chi([K/L,\mathfrak{p}])=\psi(\mathfrak{p}).$
\end{thm}

From this theorem, we obtain the following corollaries:

\begin{cor}\label{CFTCor}
Let $K/\Q$ be a Galois field extension.  Let $L\subset K$ be a subfield so that $\Gal(K/L)$ is abelian.  Let $\chi:\Gal(K/L)\rightarrow \C^*$ be a character, corresponding as described in Theorem \ref{CFTTheorem} to a Grossencharacter $\psi$ on $L$.  Then there exists a Dirichlet character $\rho$ so that
$$
\psi(\mathfrak{p}) = \rho(N_{L/\Q}(\mathfrak{p}))
$$
for all primes $\mathfrak{p}$ if and only if $\chi$ can be extended to a character on $\Gal(K/\Q)^{ab}$.
\end{cor}
\begin{proof}
Let $G=\Gal(K/\Q)$.  First we claim that if a prime $\mathfrak{p}$ of $L$ has norm $N_{L/\Q}(\mathfrak{p})=p^n$ for some rational prime $p$, then $[K/L,\mf{p}]$ is conjugate to $[K/\Q,p]^n$. This is because if the prime $\mf{q}$ of $K$ sits over $\mf{p}$, then it also sits over $p$.  If the element $g\in G$ fixes $\mf{q}$ and acts via $p$-power Frobenius on $O_K/\mf{q}$, then $g$ is in the conjugacy class of $[K/\Q,p]$.  On the other hand, $g^n$ is the unique element of $G$ that fixes $\mf{q}$ and acts on the residue field by $p^n$ power Frobenius.  Since $\mf{p}$ has residue field $\mathbb{F}_{p^n}$, this means that $g^n$ is in the conjugacy class of $[K/L,\mf{p}]$.

Suppose that $\chi$ extends to a character of $G^{ab}$, and thus to $G$.  Letting $K^{ab}$ be the maximal abelian subextension of $K$ over $\Q$, $\chi$ gives a character of $\Gal(K^{ab}/\Q)$.  Thus by Theorem \ref{CFTTheorem}, there is a Dirichlet character $\psi$ so that $\psi(p)=\chi([K^{ab}/\Q,p])=\chi([K/\Q,p])$.  We claim that for primes $\mf{p}$ of $L$ that $\chi([K/L,\mf{p}])=\psi(N_{L/\Q}(p))$.  This is because if $N_{L/\Q}(\mf{p})=p^n$ then, by the above,
$$
\chi([K/L,\mf{p}]) = \chi([K/\Q,p]^n) = \chi([K/\Q,p])^n = \psi(p)^n = \psi(p^n) = \psi(N_{L/\Q}(\mf p)).
$$

Next assume that $\chi$ is a character on $\Gal(K/L)$ and $\psi$ is a Dirichlet character so that $\chi([K/L,\mf{p}])=\psi(N_{L/\Q}(\mf{p}))$ for all $\mf{p}$.  Let $M$ be the abelian extension of $\Q$ corresponding via the correspondence in Theorem \ref{CFTTheorem} to the kernel of $\psi$.  Let $K'$ be the compositum of $M$ and $K$.  Let $G'=\Gal(K'/\Q)$.  Let $H=\Gal(K'/K)\subset G'$. By Theorem \ref{CFTTheorem}, $\psi$ corresponds to a character $\chi'$ on $\Gal(M/\Q)$, and thus to a character on $G'$. If $N_{L/\Q}(\mf{p})=p^n$, we have that
$$
\chi([K/L,\mf{p}]) = \psi(p^n) = \psi(p)^n = \chi'([K'/\Q,p])^n = \chi'([K'/\Q,p]^n) = \chi'([K'/L,\mf{p}]).
$$
By the Chebotarev Density Theorem, $[K'/L,\mf{p}]$ can take any possible value in $\Gal(K'/L)$.  Thus, for all $g\in\Gal(K'/L)$, we have that $\chi(g/H) = \chi'(g)$.  Thus $\chi'$ vanishes on $H$.  On the other hand, $\chi'$ is necessarily injective on $\Gal(M/\Q)$, and since this generates $G'$ modulo $G$, this implies that $H$ is trivial. Therefore, we have that $\chi(g)=\chi'(g)$ for $g\in \Gal(K/L)$.  Thus, $\chi'$ is an extension of $\chi$ to $G$ and thus to $G^{ab}$.
\end{proof}

\begin{cor}\label{CFTCor2}
Let $L/\Q$ be a number field and $\chi$ a Dirichlet character on $\Q$. Then $\chi(N_{L/\Q}(\mf{p}))$ is trivial on primes $\mf{p}$ of $L$ not dividing the discriminant of $\Q$, if and only if the abelian extension, $M$, of $\Q$ corresponding to the kernel of $\chi$ is contained in $L$.
\end{cor}
\begin{proof}
Let $K$ be the compositum of $M$ and the Galois closure of $L$. By Theorem \ref{CFTTheorem}, $\chi$ corresponds to some character $\psi$ of $\Gal(M/\Q)$ and thus of $\Gal(K/\Q)$. Let $\mf{p}$ be a prime of $L$ with $N_{L/\Q}(\mf{p})=p^n$. As in the proof of Corollary \ref{CFTCor}, we have that
$$
\chi(N_{L/\Q}(\mf{p}))=\chi(p^n)=\chi(p)^n=\psi([K/\Q,p])^n=\psi([K/\Q,p]^n)=\psi([K/L,\mf{p}]).
$$
By the Chebotarev Density Theorem, $[K/L,\mf{p}]$ can take on any value in $\Gal(K/L)$, and thus $\chi$ vanishes on norms from $L$ if and only if $\psi$ vanishes on this set. On the other hand, by assumption, $\psi$ is injective on $\Gal(M/\Q) = \Gal(K/\Q)/\Gal(K/M)$. Thus the kernel of $\psi$ is exactly $\Gal(K/M)$, and thus $\chi$ vanishes on norms from $L$ if and only if $\Gal(K/L)$ contains $\Gal(K/M)$, or equivalently if and only if $M$ is contained in $L$.
\end{proof}

\subsection{$G$ and $F$}\label{FGdefSec}

We begin with a standard definition:
\begin{defn}
Let $e(x)$ denote the function $e(x) = e^{2\pi i x}$.
\end{defn}

We now define $G$ as the generating function for the set primes $p\leq X$ with $[K/\Q,p]=C$ each weighted by $\log(p).$
\begin{defn}
Suppose that $K/\Q$ is a finite Galois extension with $G=\Gal(K/\Q),$ $C$ a conjugacy class of $G$, and $X$ a positive real number.  We then define the generating function
$$
G_{K,C,X}(\alpha) = \sum_{\substack{p\leq X\\ [K/\Q,p]=C}} \log(p)e(\alpha p).
$$
Where the sum is over primes, $p$, with $p\leq X$ and $[K/\Q,p]=C$.
\end{defn}

As it is a little awkward to deal with $G$ directly, we would rather work with a related function defined in terms of characters.  We first need one auxiliary definition:
\begin{defn}
Let $L/\Q$ be a number field.  Let $\Lambda_L$ be the Von Mangoldt function on ideals of $L$, $\Lambda_L:\{\textrm{Ideals of }L\}\rightarrow \R$ defined by
$$
\Lambda_L(\mf{a}) = \begin{cases} \log(N(\mf{p})) \ & \textrm{if} \ \mf{a} = \mf{p}^n\\ 0 \ & \textrm{otherwise}\end{cases}
$$
which assigns $\log(N(\mf p))$ to a power of a prime ideal $\mf p$, and 0 to ideals that are not powers of primes.
\end{defn}

We now define
\begin{defn}
If $L/\Q$ is a number field, $\xi$ a Grossencharacter of $L$, and $X$ a positive number, define the function
$$
F_{L,\xi,X}(\alpha) = \sum_{N(\mf a) \leq X} \Lambda_L(\mf a) \xi(\mf a) e(\alpha N(\mf a)).
$$
Where the sum above is over ideals $\mf a$ of $L$ with norm at most $X$.
\end{defn}

Notice that the sum in the definition of $F$ is determined up to $O(\sqrt{X})$ by the terms coming from primes $\mf a$ of prime norm.

For both $F$ and $G$, we will often suppress some of the subscripts when they are clear from context.  We now demonstrate the relationship between $F$ and $G$. One may expect them to be related since we can write the characteristic function of a conjugacy class of $G$ as a linear combination of characters induced from cyclic subgroups. The generating functions for these cyclic subgroups will turn out to give copies of $F$.

\begin{prop}\label{GFRelationProp}
Let $K$ and $C$ be as above.  Pick a $c\in C$.  Let $L\subseteq K$ be the fixed field of $c$.  Then we have that
\begin{equation}\label{GFEq}
G_{K,C,X}(\alpha) = \frac{|C|}{|G|}\left(\sum_\chi \overline{\chi}(c)F_{L,\chi,X}(\alpha) \right) + O(\sqrt{X}),
\end{equation}
where the sum is over characters $\chi$ of the subgroup $\langle c\rangle \subset G$, which, by Theorem \ref{CFTTheorem}, can be thought of as characters of $L$.
\end{prop}
\begin{proof}
We begin by considering the sum on the right hand side of Equation (\ref{GFEq}). It is equal to
\begin{align*}
\sum_\chi \overline{\chi}(c)F_{L,\chi,X}(\alpha) & =
\sum_\chi \sum_{N(\mf a)\leq X} \Lambda_L(\mf a) \overline{\chi}(c) \chi(\mf a) e(\alpha N(\mf a))\\
& = \sum_{N(\mf a)\leq X}\Lambda_L(\mf a)e(\alpha N(\mf a)) \sum_{\chi} \overline{\chi}(c)\chi([K/L,\mf a])\\
& = \ord(c)\sum_{\substack{N(\mf a)\leq X \\ [K/L,\mf a] = c}}\Lambda_L(\mf a) e(\alpha N(\mf a)).
\end{align*}
Up to an error of $O(\sqrt{X})$, we can ignore the contributions from elements whose norms are powers of primes, because there are $O(\sqrt{X}/\log(X))$ higher powers of primes with norm at most $X$.  Therefore the above equals
$$
\ord(c)\sum_{\substack{N(\mf p)\leq X \\ [K/L,\mf p] = c\\N(\mf p) \ \textrm{is prime}}}\log(N(\mf p)) e(\alpha N(\mf p))+O(\sqrt{X}).
$$

We need to determine now which primes $p\in \Z$ are the norm of an ideal $\mf p$ of $L$ with $[K/L,\mf p] = c$, and for such $p$, how many such $\mf p$ lie over it.  Each such $\mf p$ must have only one prime $\mf q$ of $K$ over it and it must be the case that $[K/\Q,\mf q]=c$.  Hence the $p$ we wish to find are exactly those that have a prime $\mf q$ lying over them with $[K/\Q,\mf q]=c$.  These are exactly the primes $p$ so that $[K/\Q,p]=C$.  Hence the term $e(\alpha n)$ appears in the above sum if and only if $n$ is a prime $p$ with $[K/\Q,p]=C$.  We next need to compute the coefficient of this term.  The coefficient will be $\ord(c)\log(p)$ times the number of primes $\mf p$ of $L$ over $p$ with $[K/\Q,\mf p]=c$.  These primes are in 1-1 correspondence with primes $\mf q$ of $K$ over $p$ with $[K/\Q,\mf q]=c$.  Now for such $p$, there will be $\frac{|G|}{\ord(c)}$ primes of $K$ over it, and $\frac{|G|}{|C|\ord(c)}$ of them will have the correct Artin symbol.  Hence the coefficient of $e(\alpha p)$ for such $p$ will be exactly $\frac{|G|}{|C|}\log(p)$.  Therefore the sum on the right hand side of Equation (\ref{GFEq}) is
$$
\frac{|G|}{|C|}\sum_{\substack{p\leq X\\ [K/\Q,p]=C}}\log(p)e(\alpha p) + O(\sqrt{X}).
$$
Multiplying by $\frac{|C|}{|G|}$ completes the proof of the Proposition.
\end{proof}

\subsection{Local Approximations}\label{LocDefSec}

Here we define some simpler functions meant to approximate $F$ and $G$.  In order to do so we will need a number of auxiliary definitions:
\begin{defn}
For $p$ a prime let
$$
\Lambda_p(n) = \begin{cases} 0 \ \ \ \ \ \ \textrm{if} \ p|n \\ \frac{1}{1-p^{-1}} \ \ \textrm{else}  \end{cases}.
$$
\end{defn}
$\Lambda_p$ can be thought of as a local approximation to the Von Mangoldt function, based only on the residue of $n$ modulo $p$.  Putting these functions together we get
\begin{defn}
Let $z$ be a positive real.  Define a function $\Lambda_z$ by
$$
\Lambda_z(n) = \prod_{p< z} \Lambda_p(n) = \begin{cases} 0 \ & \textrm{if} \ p|n \ \textrm{for some prime} \ p<z \\ \prod_{p<z} \frac{1}{1-p^{-1}} \ &\textrm{otherwise}\end{cases}.
$$
\end{defn}
Note that by slight abuse of notation we have already defined several functions denoted by $\Lambda$ with some subscript.  We will disambiguate these by context and by consistently using subscripts either the same as or nearly identical to those used in the original definition (so $\Lambda_z$ will always use $z$ as its subscript, even though this represents a variable).

There are also some related definitions which will prove useful later.
\begin{defn}
Let
$$
C(z) = \prod_{p < z} \frac{1}{1-p^{-1}}.
$$
$$
P(z) = \prod_{p < z} p.
$$
$$
P(z,q) = \prod_{p < z, p\nmid q} p.
$$
\end{defn}
We note that
$$
\Lambda_z(n) = C(z)\sum_{d|(n,P(z))} \mu(d),
$$
that
$$
\Lambda_z(n) = C(z)\sum_{d|(n,P(z,q))} \mu(d)\cdot \left(\begin{cases} 1 \ & \textrm{if} \ (n,q)=1\\ 0 \ & \textrm{otherwise} \end{cases}\right),
$$
and that
$$
C(z) = \Theta(\log(z)).
$$

We will need some other local contributions to the Von Mangoldt function to take into account splitting information.  In particular we define:
\begin{defn}
Let $K/\Q$ be a Galois extension and $C\subset G = \Gal(K/\Q)$ a conjugacy class of the Galois group. Let the image of $C$ in $G^{ab}$ correspond via Theorem \ref{CFTTheorem} to a coset $H$ of some subgroup of $(\Z/D_K\Z)^*$ for $D_K$ the discriminant of $K$.  We define $\Lambda_{K,C}$ to be the arithmetic function:
$$
\Lambda_{K,C}(n) = \begin{cases} \frac{\phi(D_K)}{|H|} \ & \textrm{if} \ n \in H\\ 0 \ & \textrm{otherwise}\end{cases}.
$$
\end{defn}
This accounts for the congruence conditions implied by $n$ being a prime with Artin symbol $C$.
\begin{defn}
Let $L$ be a number field. Let $L'$ be its maximal abelian subextension.  By Theorem \ref{CFTTheorem}, this corresponds to a subgroup $H_L$ of $(\Z/D_L\Z)^*$ for $D_L$ the discriminant of $L$. Let
$$
\Lambda_{L/\Q}(n) = \begin{cases} \frac{\phi(D_L)}{|H_L|} \ & \textrm{if} \ n\in H_L \\ 0 \ &\textrm{otherwise}\end{cases}.
$$
\end{defn}
$\Lambda_{L/\Q}$ accounts for the congruence conditions that are implied by being a norm from $L$ down to $\Q$.

We are now prepared to define our approximations $F^\sharp$ and $G^\sharp$ to $F$ and $G$.
\begin{defn}
For $K/\Q$ Galois, $C$ a conjugacy class in $\Gal(K/\Q)$, and $z$ and $X$ positive real numbers, we define the generating function
$$
G_{K,C,X,z}^\sharp(\alpha) = \frac{|C|}{|G|}\sum_{n\leq X} \Lambda_{K,C}(n)\Lambda_z(n)e(\alpha n).
$$
We also let
$$
G_{K,C,X,z}^\flat(\alpha) = G_{K,C,X}(\alpha) - G_{K,C,X,z}^\sharp(\alpha).
$$
\end{defn}
\begin{defn}
For $L/\Q$ a number field, $\xi$ a Grossencharacter of $L$, $X$ and $z$ positive numbers, we define the function
$$
F_{L,\xi,X,z}^\sharp(\alpha) =
\begin{cases}
\sum_{n\leq X} \Lambda_{L/\Q}(n) \Lambda_z(n) \chi(n) e(\alpha n) \ & \textrm{if} \ \xi = \chi\circ N_{L/\Q} \ \textrm{for some  character} \ \chi \\
0 \ & \textrm{otherwise}
\end{cases}.
$$
We note that although there may be several Dirichlet characters $\chi$ so that $\xi = N_{L/\Q}\circ\chi$, that the product of $\chi(n)$ with $\Lambda_{L/\Q}(n)$ is independent of the choice of such a $\chi$ by Corollary \ref{CFTCor2}.  We also let
$$
F_{L,\xi,X,z}^\flat (\alpha) = F_{L,\xi,X}(\alpha) - F_{L,\xi,X,z}^\sharp(\alpha).
$$
\end{defn}
Again for these functions we will often suppress some of the subscripts.

We claim that $F^\sharp$ and $G^\sharp$ are good approximations of $F$ and $G$, and in particular we will prove that:

\begin{thm}\label{GApproxThm}
Let $K/\Q$ be a finite Galois extension, and let $C$ be a conjugacy class of $\Gal(K/\Q)$.  Let $A$ be a positive integer and $B$ a sufficiently large multiple of $A$.  Then if $X$ is a positive number, $z=\log^B(X)$, and $\alpha$ any real number, then
\begin{equation}\label{GBoundEqu}
\left|G_{K,C,X,z}^\flat(\alpha) \right| = O\left(X\log^{-A}(X)\right),
\end{equation}
where the implied constant depends on $K,C,A$, and $B$, but not on $X$ or $\alpha$.
\end{thm}

\begin{thm}\label{FApproxThm}
Given $L/\Q$ a number field, and $\xi$ a Grossencharacter of $L$, let $A$ be a positive number and $B$ a sufficiently large multiple of $A$.  Then if $X$ is a positive number, $z=\log^B(X)$, and $\alpha$ any real number, then
\begin{equation}\label{FBoundEqu}
\left|F_{L,\xi,X,z}^\flat(\alpha) \right| = O\left(X\log^{-A}(X)\right),
\end{equation}
where the implied constant depends on $L,\xi,A,$ and $B$, but not on $X$ or $\alpha$.
\end{thm}

The proofs of these Theorems will be the bulk of Sections \ref{FSec} and \ref{GSec}.

\subsection{Smooth Numbers}\label{SmoothSec}

We also need some results on the distribution of smooth numbers.  We begin with a definition:
\begin{defn}
Let $S(z,Y)$ be the number of $n\leq Y$ so that $n|P(z)$.  In other words the number of $n\leq Y$ so that $n$ is squarefree and has no prime factors bigger than $z$.
\end{defn}

We will need the following bound on $S(z,Y)$:
\begin{lem}\label{SmoothNumberLem}
If $z\leq \log^B(X)$ and $Y\leq X$, then
$$S(z,Y) \ll Y^{1-1/(2B)}\exp\left( O(\sqrt{\log(X)})\right).$$
\end{lem}
\begin{proof}
Notice that
$$
\int_{y=0}^Y S(z,y)dy = \frac{1}{2\pi i} \int_{1-i\infty}^{1+i\infty} (s(s+1))^{-1} \prod_{p< z} (1+p^{-s}) Y^{s+1} ds.
$$
Note that for $\Re(s)>\frac{1}{2},$
$$
\left|\prod_{p< z}(1+p^{-s})\right| = \left|\exp\left(\sum_{p< z} p^{-s} + O(1)\right)\right| \ll \exp\left(\frac{z^{1-\Re(s)}}{1-\Re(s)} \right).
$$
Changing the line of integration to $1-\Re(s) = \frac{1}{2B}$, we get that the integrand is at most $s^{-2}Y^{2-1/(2B)}\exp\left(O(\sqrt{\log(X)}) \right)$.  Integrating and evaluating at $2Y$, we get that
$$
Y^{2-1/(2B)}\exp\left( O(\sqrt{\log(X)})\right) \gg \int_{y=0}^{2Y} S(z,y)dy \gg Y S(z,Y),
$$
proving our result.
\end{proof}

It should be noted that while Lemma \ref{SmoothNumberLem} is neither new nor the best bound currently known for $S(z,Y)$, that we use it because it is simple and sufficient for our purposes. We will also use the following Corollary.

\begin{cor}\label{SmoothNumbersCor}
If $z\leq \log^B(X)$ and $Y\geq X$, then
$$S(z,Y) \ll Y^{1-1/(3B)}.$$
\end{cor}
\begin{proof}
Apply Lemma \ref{SmoothNumberLem} with $X=Y$.
\end{proof}

\section{Approximation of $F$}\label{FSec}

In this Section, we will prove Theorem \ref{FApproxThm}.

In order to prove Theorem \ref{FApproxThm}, we will split into cases based upon whether $\alpha$ is well approximated by a rational number of small denominator.  If it is (the smooth case), we proceed to use the theory of $L$-functions to approximate $F$.  If $\alpha$ is not well approximated (the rough case), we generalize results on exponential sums over primes to show that $|F|$ is small.  In either case, $F^\sharp$ is not difficult to approximate. We note that the use of the word ``smooth" here has nothing to do with the concept of smooth numbers discussed in the previous section, and is merely an unfortunate coincidence of terminology.

We note that by Dirichlet's approximation Theorem, we can always find a pair $(a,q)$ with $a$ and $q$ relatively prime and $q<M=\Theta(X\log^{-B}(X))$ with $\left| \alpha - \frac{a}{q}\right| \leq \frac{1}{qM}$.  We consider the smooth case to be the one where $q\leq z = \log^B (X)$. As we will often be concerned with whether a real number is well approximated by a rational number of given denominator, we make the following definition:

\begin{defn}
We say that a real number $\alpha$ has a \emph{rational approximation with denominator $q$} if there exists an integer $a$ relatively prime to $q$ so that
$$
\left| \alpha - \frac{a}{q} \right| < \frac{1}{q^2}.
$$
\end{defn}

\subsection{$\alpha$ Smooth}\label{alphaSmoothSec}

In this Section, we will prove the following Proposition:
\begin{prop}\label{FrootProp}
Let $L$ be a number field, and $\xi$ a Grossencharacter.  If $z = \log^B(X)$, $Y\leq X$ and $\alpha = \frac{a}{q}$ with $a$ and $q$ relatively prime and $q\leq z$, then for some constant $c>0$ (depending only on $L$, $\xi$ and $B$),
$$
|F_{L,\xi,Y,z}^\flat(\alpha)| = O\left( X \exp\left( -c\sqrt{\log(X)}\right) \right).
$$
\end{prop}

We note that this result can easily be extended to all smooth $\alpha$.  In particular we have:
\begin{cor}\label{FsmoothCor}
Let $L$ and $\xi$ be as above.  Let $A$ be a constant, and $B$ a sufficiently large multiple of $A$.  Let $z=\log^B(X)$.  Suppose that $\alpha=\frac{a}{q}+\theta$ with $a$ and $q$ relatively prime, $q\leq z$ and $|\theta|\leq \frac{1}{qM}$ (for $M=X\log^{-B}(X)$).  Then
$$
|F_{L,\xi,X,z}^\flat(\alpha)| = O(X\log^{-A}(X)).
$$
\end{cor}
\begin{proof}[Proof (Given Proposition \ref{FrootProp})]
Noting that if $F_{L,\xi,X}^\flat (\alpha) = \sum_{n\leq X} a_n e(\alpha n)$, then by Abel summation and Proposition \ref{FrootProp},
\begin{align*}
F_{L,\xi,X}^\flat (\alpha) & = \sum_{n\leq X} a_n e\left(\frac{na}{q}\right)e\left( n\theta\right)\\
& = (1-e(\theta))\left( \sum_{Y\leq X} F_{L,\xi,Y}^\flat\left(\frac{a}{q}\right) e(Y\theta) \right) + F_{L,\xi,X}^\flat\left(\frac{a}{q}\right)e((X+1)\theta)\\
& = O\left(X^{-1}\log^{B}(X)\right)\left(\sum_{Y\leq X} O\left(X\log^{-A-B}(X) \right)\right)+O\left(X\log^{-A-B}(X) \right)\\
& = O\left(X \log^{-A}(X)\right).
\end{align*}
\end{proof}

In order to prove Proposition \ref{FrootProp}, we will need to separately approximate $F$ and $F^\sharp$.  For the former, we will also need to review some basic facts about Hecke $L$-functions.

\subsubsection{Results on $L$-functions}\label{LFunctSec}

Fix a number field $L$ and a Grossencharacter $\xi$.  We consider Hecke $L$-functions of the form $L(\xi\chi,s)$ where $\chi$ is a Dirichlet character of modulus $q\leq z=\log^B(X)$ thought of as a Grossencharacter via $\chi(\mf a) = \chi(N_{L/\Q}(\mf a))$.  We let $d$ be the degree of $L$ over $\Q$, and let $D_L$ be the discriminant.  We let $\mf m$ be the modulus of the character $\xi$, and $q$ the modulus of $\chi$.  We note that $\xi\chi$ has modulus at most $q\mf m$.  Therefore by \cite{IK}, in the paragraph above Theorem 5.35, $L(\xi\chi)$ has analytic conductor $\mf q \leq 4^d|d_L|N(\mf m) q^d$, and by Theorem 5.35 of \cite{IK}, for some constant $c$ depending only on $L$, $L(\xi\chi,s)$ has no zero in the region
$$
\sigma > 1-\frac{c}{d\log(|d_L|N(\mf m)q^d(|t|+3))}
$$
except for possibly one Siegel zero.  Note also that $L(\xi\chi,s)$ has a simple pole at $s=1$ if $\xi=\bar{\chi}$, and otherwise is holomorphic. Noting that
$$
\frac{-L'(\xi\chi,s)}{L(\xi\chi,s)} = \sum_{\mf a} \Lambda_L(\mf a)\xi\chi(\mf a)N(\mf a)^{-s},
$$
and that the $n^{-s}$ coefficient of the above is at most $d\log(n)$, we may apply Theorem 5.13 of \cite{IK} and obtain for a suitable constant $c>0$,
\begin{align}
& \sum_{N(\mf a)\leq Y}  \Lambda_L(\mf a) \xi(\mf a)\chi(\mf a) =\label{LBoundEqn}\\
& r Y - \frac{Y^\beta}{\beta} + O\left(Y\exp\left(\frac{-c\log Y}{\sqrt{\log Y}+3\log(q^d)+O(1)} \right)(\log(Yq^d)+O(1))^4 \right),\notag
\end{align}
where the term $\frac{Y^\beta}{\beta}$ should be taken with $\beta$ the Siegel zero if it exists; $r=0$ unless $\xi\chi=1$, in which case, $r=1$; and the implied constants may depend on $L$, and $\xi$ but not on $\chi$ or $Y$.

In order to make use of Equation \ref{LBoundEqn}, we will need to prove bounds on the size of Siegel zeroes.  In particular we show that:
\begin{lem}\label{SiegelZerosLem}
For all $L$ and $\xi$, and all  $\epsilon>0$, there exists a $c(\epsilon)>0$ so that for every Dirichlet character $\chi$ of modulus $q$ and every Siegel zero $\beta$ of $L(\xi\chi,s)$,
$$
\beta > 1- \frac{c(\epsilon)}{q^\epsilon}.
$$
\end{lem}
\begin{proof}
We follow the proof of Theorem 5.28 part 2 from \cite{IK}, and note the places where we differ.  We note that Theorem 5.35 states that we only need by concerned with the case when $\xi\chi$ is totally real.  We then consider two such $\chi$ having Siegel zeros.  We use, $L(s)=\zeta_L(s)L(\xi\chi_1,s)L(\xi\chi_2,s)L(\xi^2\chi_1\chi_2)$, which has conductor $O(q_1 q_2)^{2d}$ instead of the analogous one from \cite{IK}.  This gives us a convexity bound on the integral term of $O((q_1 q_2)^d x^{1-\beta})$, instead of the one listed.  Again assuming that $\beta > 3/4$, we take $x>c(q_1 q_2)^{4d}$.  We notice that we still have (5.64) for $\sigma>1-1/d+\epsilon$ (for any $\epsilon>0$) by noting that $|\sum_{N(\mf a)\leq x}\xi\chi(\mf a)|=O(x^{1-1/d}+\max(x,q))$.  Therefore, Equation (5.75) of \cite{IK} becomes
$$
L(\xi\chi_2,1) \gg (1-\beta_1)(q_1q_2)^{-4d(1-\beta_1)}(\log (q_1 q_2))^{-2}.
$$
The rest of the argument from \cite{IK} carries over more or less directly.
\end{proof}

\subsubsection{Approximation of $F$}\label{FSmoothApproxSec}

We prove
\begin{prop}\label{FatRatProp}
With $L,\xi,\chi,Y,r$ as above, $X\geq Y$ and $z=\log^B(X)$,
\begin{equation}\label{PrimeSumEq}
F_{L,\xi\chi,Y}(0) = rY+ O\left(X\exp(-c\sqrt{\log(X)})\right).
\end{equation}
Where again $c$ depends on $L,\xi$ but not $\chi,X,Y$.
\end{prop}
\begin{proof}
Applying Lemma \ref{SiegelZerosLem} with $\epsilon=1-1/(2B)$ to Equation \ref{LBoundEqn}, we get that
\begin{align*}
\sum_{N(\mf a)\leq Y}  \Lambda_L(\mf a) \xi(\mf a)\chi(\mf a) & = r Y - \frac{Y^\beta}{\beta} + O\left( X \exp(-c\sqrt{\log(X)})\right)\\
& = r Y + O\left( Y \exp(-c(\epsilon)\sqrt{\log(Y)}\right) + O\left( X \exp(-c\sqrt{\log(X)})\right)\\
& = r Y + O\left( X \exp(-c\sqrt{\log(X)})\right).
\end{align*}
\end{proof}

\subsubsection{Approximation of $F^\sharp$}\label{FSharpSmoothApproxSec}

\begin{prop}\label{FSharpAtRatProp}
With $L,\xi,\chi,X,Y,r$ as above, $z=\log^B(X)$,
\begin{equation*}
F_{L,\xi\chi,Y}^\sharp(0) = rY+ O\left(X\exp(-c\sqrt{\log(X)})\right).
\end{equation*}
\end{prop}
\begin{proof}
If $\xi\chi$ is not of the form $\chi'\circ N_{L/\Q}$, then $F^\sharp = 0$ and we are done.  Otherwise let $\xi\chi$ be as above with $\chi'$ a character of modulus $q'$.  We have that
\begin{align*}
F_{L,\chi',Y}^\sharp(0) & = \sum_{n\leq Y} \Lambda_{L/\Q}(n)\Lambda_z(n)\chi'(n)\\
& = C(z) \sum_{n\leq Y} \sum_{d|(P(z,q'D_L),n)}\mu(d)\Lambda_{L/\Q}(n)\chi'(n)\\
& = C(z) \sum_{d|P(z,q'D_L)} \sum_{n=dm \leq Y} \mu(d)\Lambda_{L/\Q}(n)\chi'(n)\\
& = C(z) \sum_{d|P(z,q'D_L)}\mu(d)\chi'(d)\sum_{m\leq Y/d} \Lambda_{L/\Q}(dm)\chi'(m).
\end{align*}
Consider for a moment the inner sum over $m$.  It is periodic with period $q'D_L$.  Note that the sum over a period is 0 unless $\chi'$ is trivial on $H_L$, in which case the average value is $\overline{\chi'}(d)\frac{\phi(q'D_L)}{q'D_L}$.  Since $r=1$ if $\chi'$ vanishes on $H_L$ and $r=0$ otherwise, we have that:
\begin{align*}
F_{L,\chi',Y}^\sharp(0) & = C(z)\left(\frac{\phi(q'D_L)}{q'D_L}\right)\sum_{\substack{d|P(z,qD_L)\\ d\leq Y}}\left(\frac{r\mu(d)Y}{d} + O(q'D_L)\right).
\end{align*}
The sum of error term here is at most $O\left( C(z)q S(z,Y)\right)$ which by Lemma \ref{SmoothNumberLem} is $ O\left(Y^{1-1/(2B)}\log^2(z)q \exp(O(\sqrt{\log(X)})) \right).$  The remaining term is
$$
r Y C(z)\frac{\phi(q'D_L)}{q'D_L}\sum_{\substack{d|P(z,q'D_L)\\ d\leq Y}} \frac{\mu(d)}{d}.
$$
The error introduced by extending the sum to all $d|P(z,q'D_L)$ is at most
$$O\left(YC(z)\int_Y^\infty S(z,y)y^{-2}dy \right).$$  By Lemma \ref{SmoothNumberLem} this is $$O\left(Y^{1-1/(2B)}\log(z)\exp(O(\sqrt{\log(X)}))\right).$$  Once we have extended the sum we are left with
\begin{align*}
r Y C(z)\frac{\phi(q'D_L)}{q'D_L}\sum_{\substack{d|P(z,q'D_L)}} \frac{\mu(d)}{d} & =
r Y C(z) \left(\frac{\phi(q'D_L)}{q'D_L}\right)\left( \frac{\phi(P(z,q'D_L))}{P(z,q'D_L)}\right) \\
& = r Y C(z) \left(\frac{\phi(P(z))}{P(z)} \right)\\
& = r Y.
\end{align*}
Hence
\begin{align*}
F_{L,\xi\chi,Y,z}^\sharp(0) & = rY + O\left(Y^{1-1/(2B)}\log^2(z)q^2 \exp(O(\sqrt{\log(X)})) \right)\\
& = rY + O\left(X\exp(-c\sqrt{\log(X)})\right).
\end{align*}
\end{proof}

\subsubsection{Proof of Proposition \ref{FrootProp}}

\begin{proof}
Combining Propositions \ref{FatRatProp} and \ref{FSharpAtRatProp} we obtain that
$$
F_{L,\xi\chi,Y,z}^\flat(0) = O\left(X\exp(-c\sqrt{\log(X)})\right).
$$
Our Proposition follows immediately after noting that
$$
F_{L,\xi,X,z}^\flat\left(\frac{a}{q}\right) = \frac{1}{\phi(q)}\sum_{\chi \textrm{ mod } q} G(\bar{\chi},a/q) F_{L,\xi\chi,X,z}^\flat(0).
$$
Where $G(\bar{\chi},a/q)$ is the Gauss sum
$$
G(\bar{\chi},a/q) = \sum_{x\pmod q} \bar{\chi}(x)e(ax/q).
$$
\end{proof}

\subsection{$\alpha$ Rough}\label{alphaRoughSec}

In this Section, we will show that $|F^\flat(\alpha)|$ is small for $\alpha$ not well approximated by a rational of small denominator.  We will do this by showing that both $|F(\alpha)|$ and $|F^\sharp(\alpha)|$ are small.  The proof of the latter will resemble the proof of Proposition \ref{FSharpAtRatProp}.  The proof of the former will require some machinery including some Lemmas about rational approximations and exponential sums of polynomials.

\subsubsection{Bounds on $F^\sharp$}\label{FSharpRoughSec}

\begin{prop}\label{FsharpRoughProp}
Fix $L$ a number field, and $\xi$ a Grossencharacter.  Fix $B$ and let $z=\log^B(X)$.  Let $\alpha$ be a real number.
If $\alpha$ has a rational approximation with denominator $q$, then
$$
|F_{L,\xi,z}^\sharp(\alpha)| = O\left(X\log(X)\log(z)q^{-1} +q\log(q)\log(z) + X^{1-1/(4B)}\exp(O(\sqrt{\log(X)})\right),
$$
where the implied constant may depend on $L$ and $\xi$ but nothing else.
\end{prop}
\begin{proof}
We note that the result is trivial unless $\xi=N_{L/\Q}(\chi)$ for some Dirichlet character $\chi$ of modulus $Q$.  Hence we may assume that
$$
F_{L,\xi,z}^\sharp(\alpha) = \sum_{n\leq X}\Lambda_{L/\Q}(n)\Lambda_z(n)\chi(n)e(\alpha n).
$$
Let $D_L$ be the discriminant of $L$.
We note that
\begin{align*}
F_{L,\xi,z}^\sharp(\alpha) & = \sum_{n\leq X} \Lambda_{L/\Q}(n)\Lambda_z(n)\chi(n)e(\alpha n)\\
& = C(z) \sum_{n\leq X} \sum_{d|(n,P(z,QD_L))} \mu(d) \Lambda_{L/\Q}(n)\chi(n) e(\alpha n)\\
& = C(z) \sum_{d|P(z,QD_L)} \mu(d)\chi(d) \sum_{md = n \leq X} \Lambda_{L/\Q}(dm)\chi(m)e(\alpha dm)\\
& = O\left( C(z) \sum_{d|P(z,QD_L)} \left| \sum_{m\leq X/d} \Lambda_{L/\Q}(dm)\chi(m)e(\alpha dm) \right|\right).
\end{align*}
In order to analyze the last sum, we split it up based on the residue class of $m$ modulo $QD_L$.  Each new sum is a geometric series with ratio of terms $e(\alpha QD_L d)$.  Hence we can bound this sum as $\min\left(\frac{X}{d},\frac{QD_L}{2||dQD_L\alpha||}\right)$, where $||x||$ is the distance from $x$ to the nearest integer. Therefore we have that
$$
|F_{\chi,z}^\sharp(\alpha)| = O\left( C(z) \sum_{d\leq X^{1-1/(4B)}} \min\left(\frac{X}{d},\frac{QD_L}{2||dQD_L\alpha||}\right) + C(z) X^{1/4B} S(z,X)\right).
$$
We bound the sum in the first term by looking at what happens as $d$ ranges over an interval of length $\frac{q}{3QD_L}$. We get that $dQD_L\alpha = x_0 + kQ\alpha$ for $x_0$ the value at the beginning of the interval and $k$ an integer at most $\frac{q}{3QD_L}$.  Notice that $kQD_L\alpha$ is within $\frac{1}{3q}$ of $\frac{kQD_La}{q}$, which must be distinct for different values of $k$.  Hence none of the fractional parts of $dQD_L\alpha$ can be within $\frac{1}{3q}$ of each other. Hence the sum over this range of $d$ is at most $\frac{X}{d} + \frac{2QD_L}{2/(3q)} + \frac{2QD_L}{2(2/2q)} + \ldots =O\left( \frac{X}{d} + 3qQD_L\log(q)\right).$  Furthermore the $\frac{X}{d}$ term does not show up in the first such interval, since when $d=0$, $dQD_L\alpha$ is an integer. We have $3QD_LX^{1-1/(4B)}/q+1$ of these intervals.  Therefore, the first term is at most $C(z)$ times
\begin{align*}
O&\left(  \frac{X}{(q/(3QD_L))} + \frac{X}{2(q/(3QD_L))} + \ldots + 9Q^2D_L^2\log(q) X^{1-1/(4B)}+3qQD_L\log(q)  \right) \\
& = O\left( X \log(X)q^{-1} + \log(q)X^{1-1/(4B)} + q\log(q)\right).
\end{align*}
The other term is bounded by Lemma \ref{SmoothNumberLem} as
$$
O\left( \log(z) X^{1-1/(4B)}\exp\left(O(\sqrt{\log(X)}) \right)\right).
$$
Putting these bounds together, we get that
$$
|F_{L,\xi,z}^\sharp(\alpha)| = O\left(X\log(X)\log(z)q^{-1} +q\log(q)\log(z) + X^{1-1/(4B)}\exp(O(\sqrt{\log(X)})\right).
$$
\newline
\end{proof}

\subsubsection{Lemmas on Rational Approximation}\label{RatApproxSec}

In the coming Sections, we will need some results on rational approximation of numbers.  In particular, we will need to know how often multiples of a given $\alpha$ have a good rational approximation. We have the following Lemmas.
\begin{lem}\label{RatAproxMultLem}
Let $X,Y,A$ be positive integers.  Let $\alpha$ be a real number with rational approximation of denominator $q$.  Suppose that for some $B$, that $XY B^{-1} > q > B$.  Then for all but $O\left(Y\left(A^{3/2}B^{-1/2} + A^2 B^{-1} +\log(AY)A^3 X^{-1} \right)\right)$ of the integers $n$ with $1\leq n \leq Y$, $n\alpha$ has a rational approximation with denominator $q'$ for any $XA^{-1} > q' > A$.
\end{lem}
\begin{proof}
By Dirichlet's approximation theorem, $n\alpha$ always has a rational approximation $\frac{a}{q'}$ with $q'<XA^{-1}$ and
$$
\left| n\alpha - \frac{a}{q'} \right| < \frac{1}{q'XA^{-1}}.
$$
Therefore, $n\alpha$ lacks an appropriate rational approximation only when the above has a solution for some $q'\leq A$.  If such is the case then, dividing by $n$, we find that $\alpha$ is within $(q')^{-1}n^{-1}X^{-1} A$ of some rational number of denominator $d$ so that $d|nq'$.  Note that this error is at most $\max(n,d)^{-1}X^{-1} A$.

Given such a rational approximation to $\alpha$ with denominator $d$, we claim that it contributes to at most $YA^2 d^{-1}$ bad $n$'s.  This is because there are at most $A$ values of $q'$, and for each value of $q'$, we still need that $n$ is a multiple of $\frac{d}{(d,q')} \geq d A^{-1}$.  Hence for each $q'$, there are at most $YAd^{-1}$ bad $n$. Since there are at most $A$ values of $q'$, we have at most $YA^2d^{-1}$ bad $n$.

Next, we pick an integer $n_0$.  We will now consider only $Y\geq n \geq n_0$ so that $\alpha n$ has no suitable rational approximation.  We do this by analyzing the denominators $d$ for which some rational number of denominator $d$ approximates $\alpha$ to within $X^{-1} A (\max(d,n_0))^{-1}.$  Suppose that we have some $d\neq q$ which does this.  $\alpha$ is within $q^{-2}$ of a number with denominator $q$, and within $X^{-1}n_0^{-1} A$ of one with denominator $d$.  These two rational numbers differ by at least $(dq)^{-1}$ and therefore,
$$
(dq)^{-1} \leq q^{-2} + X^{-1} A n_0^{-1}.
$$
Hence, either $dq^{-1}$ or $X^{-1}A n_0^{-1} dq$ is at least $\frac{1}{2}.$  Hence, either $d\geq\frac{q}{2}$, or
$$
d \geq \frac{X n_0}{2 A q} \geq \frac{n_0 B}{2 A Y}.
$$
Therefore, the smallest such $d$ is at least the minimum of $\frac{q}{2}$ and $\frac{n_0 B}{2 A Y}$.

Next, suppose that we have two different such denominators, say $d$ and $d'$.  The fractions they represent are separated by at least $(dd')^{-1}$ and yet are both close to $\alpha$.  Therefore,
$$
(dd')^{-1} \leq X^{-1} A (d^{-1} + d'^{-1}).
$$
Therefore, we have that $\max(d,d') \geq \frac{X}{2A}.$  Hence, there is at most one such denominator less than $\frac{X}{2A}$.

Next, we wish to bound the number of such denominators $d$ in a dyadic interval $[K,2K]$.  We note that the corresponding fractions are all within $X^{-1}AK^{-1}$ of $\alpha$, and that any two are separated from each other by at least $(2K)^{-2}$.  Therefore, the number of such $d$ is at most $1+8KX^{-1}A$.

To summarize we potentially have the following $d$ each giving at most $YA^2d^{-1}$ bad $n$'s.
\begin{itemize}
\item One $d$ at least $\min\left(\frac{q}{2},\frac{n_0 B}{2AY}\right)$.
\item For each diadic interval $[K,2K]$ with $K\geq \frac{X}{2A}$ at most $10KX^{-1}A$ such $d$'s
\end{itemize}
Notice that there are $\log(2AY)$ such diadic intervals, and that each contributes at most $10YA^3X^{-1}$ bad $n$'s.
We also potentially have $n_0$ bad $n$'s from the numbers less than $n_0$.  Hence the number of $n$ for which there is no suitable rational approximation of $n\alpha$ is at most
$$
 O\left( n_0 + YA^2 B^{-1} + Y^2 A^3 B^{-1} n_0^{-1} + \log(AY)YA^{3} X^{-1}\right).
$$
Substituting $n_0 = Y A^{3/2} B^{-1/2}$ yields our result.
\end{proof}

We will also need the following related Lemma:
\begin{lem}\label{RatApproxStructureLem}
Let $X,A,C$ be positive integers.  Let $\alpha$ be a real number with rational approximation of denominator $q$.  Suppose that for some $B>2A$, that $XB^{-1} > q > B$.  Then there exists a set $S$ of natural numbers so that
\begin{itemize}
\item elements of $S$ are of size at least $\Omega(BA^{-1})$.
\item The sum of the reciprocals of the elements of $S$ is $O(A^2 B^{-1}+X^{-1}A^4 C)$.
\item for all positive integers $n\leq C$, either $n$ is a multiple of some element of $S$ or $n\alpha$ has a rational approximation with some denominator $q'$ with $XA^{-1}n^{-1} > q' > A$.
\end{itemize}
\end{lem}
\begin{proof}
We use the same basic techniques as the proof of Lemma \ref{RatAproxMultLem}.

We begin by letting $S$ be the set of all integers of the form $\frac{d}{D}$ for some integers $d,D$ with $A\geq D$, $D|d$, $d\leq AC$ and
$$
\left|\alpha - \frac{a}{d} \right| \leq  \frac{1}{XA^{-1}},
$$
for some integer $a$ relatively prime to $d$.

We begin by verifying the third claim for this set $S$. Note that $n\alpha$ always has a rational approximation $\frac{a}{q'}$ accurate to within $\frac{1}{q'XA^{-1}n^{-1}}$ with $q'<XA^{-1}n^{-1}$.  This means that we have an appropriate rational approximation of $n\alpha$ unless this $q'$ is less than $A$. If this happens, it is the case that
$$
\left|\alpha - \frac{a}{nq'} \right| \leq \frac{1}{q'XA^{-1}} \leq \frac{1}{XA^{-1}}.
$$
Letting $d=nq'/\gcd(a,nq')\leq AC$ and $D=q'/\gcd(a,q')$, we see that $n$ is a multiple of $\frac{d}{D}$, which is in $S$ since
$$
\left|\alpha - \frac{a/\gcd(a,nq')}{d} \right| \leq \frac{1}{XA^{-1}}.
$$

To verify the first property, we note that if we have integers $a$ and $d$, with $d$ not a multiple of $q$, so that
$$
\left|\alpha - \frac{a}{d} \right| \leq  \frac{1}{XA^{-1}},
$$
then $\alpha$ is within $q^{-2}$ of a rational number of denominator $q$ and within $X^{-1}A$ of one of denominator $d$.  Hence,
$$
(dq)^{-1} \leq q^{-2} + X^{-1}A.
$$
Therefore,
$$
d\geq \min\left( \frac{q}{2},\frac{X}{2A}\right)\geq \frac{q}{2}.
$$
Therefore, every element of $S$ is of the form $\frac{d}{D}$ with $d\geq \frac{q}{2} \geq \frac{B}{2}$ and $D\leq A$. Thus every such element is $\Omega(BA^{-1})$.

Finally, we verify the second property. To each element $\frac{d}{D}$ of $S$, we may associate the rational number $\frac{a}{d}$ so that
$$
\left|\alpha - \frac{a}{d} \right| \leq  \frac{1}{XA^{-1}}.
$$
The sum of the reciprocals of elements of $S$ associated to this fraction is at most $\sum_{D\leq A} (d/D)^{-1} = O(A^2d^{-1})$. Given two such approximations, $\frac{a}{d}$ and $\frac{a'}{d'}$, they must differ by at most $\frac{2}{XA^{-1}}$, and thus
$$
(dd')^{-1} \leq \frac{2}{XA^{-1}}.
$$
Therefore the second largest such $d$ is at least $\sqrt{XA^{-1}/2}.$

Next we consider the number of such approximations with $d$ lying in a diadic interval $[K,2K]$. All of these approximations are within $X^{-1}A$ of $\alpha$ and are separated from each other by at least $\frac{1}{4K^2}$. Therefore, taking $K>\sqrt{XA^{-1}/2}$, the number of such approximations is $O(X^{-1}AK^2)$, so the contribution they make to the sum of the reciprocals of the elements of $S$ is at most $O(X^{-1}A^3 K)$. Summing this over $K$ a power of 2 of size at most $AC$, yields a total contribution of $O(X^{-1}A^4 C)$. Thus, the sum of the reciprocals of elements of $S$ corresponding to all of the appropriate rational approximations except for the one of minimal denominator is at most $O(X^{-1}A^4 C)$. The contribution coming from the approximation with minimal denominator consists of the sum of reciprocals of $O(A)$ terms each of size $\Omega(B/A)$, and is thus $O(A(B/A)^{-1})=O(A^2 B^{-1})$. Combining this with the contribution from the other rational approximations yields our result.
\end{proof}

We will be using Lemma \ref{RatApproxStructureLem} to bound the number of ideals of $L$ so that $N(\mf a)\alpha$ has a good rational approximation.  In order to do this we will also need the following:

\begin{lem}\label{NormMultLem}
Fix $L$ be a number field.  Let $n$ be a positive integer, and let $X$ and $\epsilon$ be positive real numbers.  Then for any $\epsilon>0$, we have that:
$$
\sum_{\substack{n | N(\mf a)\\N(\mf a) < X}} \frac{1}{N(\mf a)} = O \left(\frac{\log(X) n^\epsilon}{n} \right),
$$
$$
\sum_{\substack{n | N(\mf{ab})\\N(\mf{ab}) < X}} \frac{1}{N(\mf{ab})} = O \left(\frac{\log^2(X) n^\epsilon}{n} \right),
$$
where the implied constant depends on $L$ and $\epsilon$, but nothing else.
(The first sum above is over ideals $\mf a$ so that $n|N(\mf a)$ and $N(\mf a)\leq X$, the second over pairs of ideals $\mf a$ and $\mf b$, so that $N(\mf a \cdot \mf b)$ satisfies the same conditions).
\end{lem}
\begin{proof}
We will prove the first of the two equations and note that the second follows from a similar argument.  Let $d=[L:\Q]$.  Let $p_1,\ldots,p_k$ be the distinct primes dividing $n$.  We claim that for such an ideal $\mf a$ must be a multiple of some ideal $\mf a_0$ with $N(\mf a_0)=nm$ with $m=\prod_{i=1}^k p_i^{a_i}$ for some $0\leq a_i < d$.  We obtain this by starting with the ideal $\mf a_0=(1)$ and repeatedly multiplying by primes of $\mf a/\mf a_0$ whose norm is a power of one of the $p_i$ that do not yet divide $N(\mf a_0)$ sufficiently many times.  Since this prime has norm no bigger than $p_i^d$ we cannot overshoot by more than $d-1$ factors of any $p_i$.  We note that the number of possible values of $m$ is $k^d$.  Since $k=O(\log(n))$ this is $O(n^{\epsilon/2})$.  For each value of $m$ there are $O(n^{\epsilon/2})$ ideals of norm exactly $nm$, and hence there are $O(n^\epsilon)$ possible ideals $\mf a_0$.

Thus we have
\begin{align*}
\sum_{\substack{n | N(\mf a)\\N(\mf a) < X}} \frac{1}{N(\mf a)} & = \sum_{\mf{a}_0} \sum_{N(\mf{b})\leq X/N(\mf{a}_0)} \frac{1}{N(\mf{a}_0\mf{b})}\\
& \leq \sum_{\mf{a}_0} \frac{1}{N(\mf{a}_0)} \sum_{N(\mf{b}) \leq X} \frac{1}{N(\mf{b})}\\
& \ll \sum_{\mf{a}_0} \frac{\log(X)}{n}\\
& \ll \frac{\log(X) n^{\epsilon}}{n}.
\end{align*}
\end{proof}

\subsubsection{Lemmas on Exponential Sums}\label{ExpSumsSec}

We will need a Lemma on the size of exponential sums of polynomials along the lines of Lemma 20.3 of \cite{IK}.  Unfortunately, the $X^\epsilon$ term that shows up there will be unacceptable for our application.  So instead we prove:

\begin{lem}\label{polySumBoundLem}
Pick a positive integer $X$.  Let $[X] = \{1,2,\ldots,X\}$.  Let $P$ be a polynomial with leading term $c x^k$ for some integer $c\neq 0$.  Let $\alpha$ be a real number with a rational approximation of denominator $q$.  Then
$$
\left| \sum_{x\in [X]} e(\alpha P(x))\right| \ll |c| X \left(\frac{1}{q} + \frac{1}{X} + \frac{q}{X^k} \right)^{10^{-k}},
$$
where the implied constant depends on $k$, but not on the coefficients of $P$.
\end{lem}
Note that the $10^{-k}$ in the exponent is not optimal and was picked for convenience.
\begin{proof}
We proceed by induction on $k$.  We take as a base case $k=1$.  Then we have that $P$ is a linear function with linear term $c$.  $\alpha$ is within $q^{-2}$ of a rational number of denominator $q$.  Therefore $c\alpha$ is within $c q^{-2}$ of a number of denominator between $q c^{-1}$ and $q$.  If $c\geq q/2$, there is nothing to prove.  Otherwise, $c\alpha$ cannot be within $q^{-1} - cq^{-2} = O(q^{-1})$ of an integer.  Therefore the sum is at most $O(\min(X,q))$, which clearly satisfies the desired inequality.

We now perform the induction step.  We assume our inequality holds for polynomials of smaller degree.  Squaring the left hand side of our inequality, we find that
\begin{align*}
\left| \sum_{x\in [X]} e(\alpha P(x))\right| & = \left(\sum_{a,b\in [X]} e(\alpha(P(a)-P(b))) \right)^{1/2}.
\end{align*}
Breaking the inner sum up based on the value of $n=a-b$, we note that $P(n+b)-P(b)$ is a polynomial in $b$ of degree $k-1$ with leading term $nckx^{k-1}$.  Letting $[X_n]=\{1,2,\ldots,X\}\cap\{1-n,2-n,\ldots,X-n \}$, we are left with at most
\begin{align*}
&\left( \sum_{n\in [-X,X]} \left|  \sum_{b\in [X_n]} e(\alpha(P(b+n)-P(b)))\right|\right)^{1/2}
\\  = & \left( \sum_{n\in [-X,X]} \left|\sum_{b\in [X_n]} e\left((n\alpha)\left(\frac{P(b+n)-P(b)}{n}\right)\right)\right|\right)^{1/2}.
\end{align*}

Let $B=\min(q,X^k/q)$.
We consider separately the terms in the above sum where $n\alpha$ has no rational approximation with denominator between $B^{1/5}$ and $X^{k-1}B^{-1/5}$. By Lemma \ref{RatAproxMultLem} with parameters $A=B^{1/5},B=B,Y=X,X=X^{k-1},$ the number of such $n$ is at most $O(X(B^{-1/5} +\log(X)B^{3/5}X^{1-k}))$.  Each of those terms contributes $O(X)$ to the sum and hence together they contribute at most
$$
O(X(B^{-1/10} + \log(X)B^{3/10}X^{(1-k)/2})).
$$
Which is within the required bounds.

For the other terms, the inductive hypothesis tells us that the sum for fixed $n$ is at most
$$
O\left(|c|(X-n)\left(B^{-1/5} + \frac{1}{X-|n|} + \frac{X^{k-1}B^{-1/5}}{(X-|n|)^{k-1}}\right)^{-10^{k-1}}\right).
$$
Summing over $n$ and taking a square root gives an appropriate bound.
\end{proof}

We apply this Lemma to get a bound on exponential sums of norms of ideals of a number field.  In particular we show that:
\begin{lem}\label{IdealSumLem}
Fix $L$ a number field of degree $d$, and $\xi$ a Grossencharacter of modulus $\mf m$.  Then given a positive number $X$ and a real number $\alpha$ which has a rational approximation of denominator $q$, we have that
\begin{equation}\label{IdealSumEqn}
\left|\sum_{N(\mf a) \leq X} \xi(\mf a) e(\alpha N(\mf a)) \right| = O\left(X\left(\frac{1}{q} +\frac{1}{X^{1/d}} +\frac{q}{X} \right)^{10^{-d}/2} \right).
\end{equation}
Where the implied constant depends only on $L$ and $\xi$.
\end{lem}
\begin{proof}
We begin by dividing the sum in question into pieces based on the class of $\mf{a}$ modulo multiplication by elements of $\mathcal{O}_L$ congruent to $1$ modulo $\mf{m}$. It is well known that there are only finitely many such classes, and thus it suffices to show that for any such class, the sum over $\mf{a}$ in that class is bounded by the right hand side of Equation \eqref{IdealSumEqn}. We will henceforth proceed to bound the sum over $\mf{a}$ in one such class.

Pick a representative ideal $\mf{a}_0$ of the class in question. Let $L_{\mf m}^1$ denote the set of elements of $L$ congruent to $1$ modulo $\mf{m}$. Every ideal in our class can be written in the form $b\mf{a}_0$ for some $b\in L_{\mf m}^1\cap \mf{a}_0^{-1}$ (where $\mf{a}_0^{-1}$ is the appropriate fractional ideal). Furthermore, this representation is unique up to multiplying $b$ by an element of $\mathcal{O}_L^*\cap L_{\mf m}^1$. We have that $\xi(b\mf{a}_0) = \xi(b)\xi(\mf{a_0})$. Since $\xi$ has modulus $\mf m$ and since $b\in L_{\mf m}^1$, $\xi(b)=\psi(b)$ for $\psi$ some continuous character $\psi:(L\otimes\R)^* \rightarrow \C^*$, with $\psi(\mathcal{O}_L^*\cap L_{\mf m}^1)=1$. Additionally, we have that $N_{L/\Q}(b\mf{a}_0) = |N_{L/\Q}(b)|N_{L/\Q}(\mf{a}_0)$.

We simplify the sum in question by considering the geometry of the set of $b$'s in question. In particular, we note that $T:=L_{\mf m}^1\cap \mf{a}_0^{-1}$ is a translate of a lattice in $L\otimes \R$. Furthermore $N_{L/\Q}(b)$ is easily seen to be a degree $d$ polynomial with rational coefficients on this lattice. The sum that we wish to take is over all $b$ in this lattice with norm at most $X/N_{L/\Q}(\mf{a}_0)$ in some fundamental domain of the action of $\mathcal{O}_L^*\cap L_{\mf m}^1$. We can obtain such a region by letting $D$ be a fundamental domain of $\mathcal{O}_L^*\cap L_{\mf m}^1$ within the set of unit norm elements of $(L\otimes \R)^*$. We can then take our sum to be over all $b$ in $R:=D\cdot (0,(X/N_{L/\Q}(\mf{a}_0))^{1/d}]$. By Dirichlet's Unit Theorem, $D$ can be taken to be a bounded region with finite volume and surface area and finitely many connected components within the unit norm elements of $(L\otimes \R)^*$ (which by taking logarithms is isomorphic to a torus times some number of copies of $\R$, giving us notions of volume and surface area). For such $D$, it is easy to see that $R$ produces a region in $L\otimes \R$ with volume $\Theta(X)$ and surface area $O(X^{1-1/d})$ (where the implied constant depends on our choice of $D$) and finitely many connected components. Summarizing the above, we find that the expression that we need to bound is
$$
\xi(\mf{a}_0)\sum_{b\in T\cap R} \psi(b)e(\alpha |P(b)|),
$$
where $P$ is some polynomial of degree $d$ on $T$ with rational coefficients. It should be noted that $P(b)$ will have constant sign on connected components of $R$ (since $P$ extends to a non-zero, continuous function on $R$). Thus, by restricting our sum to a single connected component of $R$, we may ignore the absolute value of $P$ taken above.

In order to reduce the above sum to something that can be handled with Lemma \ref{polySumBoundLem}, we need to reduce to the case where we are summing $e(\alpha p(x))$ for $p$ some polynomial in one variable with integer leading term. In order to do this, we pick some non-zero vector $v$ under which $T$ is translation invariant. Then for any $b\in T$, $P(b+nv)$ is a degree-$d$ polynomial in $n$ whose rational, leading coefficient does not depend on $b$. Perhaps replacing $v$ by a positive multiple of itself, we may assume that this leading coefficient is a non-zero integer. Fix an integer $Y=\Theta(X^{1/(2d)})$. Define a \emph{line} in $T$ to be a subset of $T$ of the form $\{b+v,b+2v,\ldots,b+Yv\}$ for some $b\in T$. Each element of $T$ is contained in exactly $Y$ lines, and thus the sum in question can be written as
$$
\frac{\xi(\mf{a}_0)}{Y} \sum_{\textrm{lines }N}\sum_{b\in N\cap R} \psi(b)e(\alpha P(b)).
$$

We break the above sum into cases based upon whether or not $N$ is contained in $R$. If $N$ is neither contained in $R$ nor disjoint from $R$, then it must be defined by some $b\in T$ so that $b+xv\in \partial R$ for some $x\in [0,Y]$. This means that $b$ must lie within distance $O(Y)$ of the boundary of $R$. Extending a fundamental domain of $T$ around each such $b$, we find that their union is contained in a set of volume $O(X^{1-1/d}Y)$, and thus there are at most $O(X^{1-1/d}Y)$ such lines. Each such line contributes $O(1)$ to the above sum, thus the total contribution to the above sum coming from lines $N$ not contained in $R$ is $O(X^{1-1/d}Y)$, which is within the desired bounds.

Next consider the contribution coming from a particular line $N$ contained in $R$. We note that each of these points corresponds to a $b+v\in T\cap R$, and thus that there are at most $O(X)$ of these lines. Let $N=\{b+v,\ldots,b+Yv\}$. The sum in question over this line reduces to
$$
\sum_{n=1}^Y \psi(b+nv)e(\alpha P(b+nv)) = \psi(b+v) \sum_{n=1}^Y \psi(1+(n-1)v(b+v)^{-1}) e(\alpha p_b(n)),
$$
where $p_b(n)$ is a degree $d$ polynomial in $n$ with coefficients dependent on $b$, but whose leading term is integral and does not depend on $b$. Since $\psi$ is continuous (and thus smooth) we may write $\psi(1+nvb^{-1}) = 1+O(Y|(b+v)^{-1}|)$, where $|(b+v)^{-1}|$ is the maximum of the absolute value of $(b+v)^{-1}$ at any of the infinite places (where for complex places, we use the standard absolute value rather than its square). Thus the absolute value of the sum in question is
$$
O(\min(Y^2|(b+v)^{-1}|,Y))+\sum_{n=1}^Y e(\alpha p_b(n)).
$$
By Lemma \ref{polySumBoundLem}, the latter term above is
$$
O\left(Y\left(\frac{1}{q}+\frac{1}{Y}+\frac{q}{Y^d}\right)^{-10^{-d}}\right).
$$
Summing this latter term over all lines contained in $R$, gives a contribution to our final sum of size
$$
O\left(X\left(\frac{1}{q}+\frac{1}{Y}+\frac{q}{Y^d}\right)^{-10^{-d}}\right),
$$
which is of the appropriate size.

We have left to bound the sum over lines $N$ contained in $R$ of $O(\min(Y^2|(b+v)^{-1}|,Y))$. This is at most the sum over elements $a\in T\cap R$ of $O(\min(Y^2|a^{-1}|,Y))$. This in turn is
\begin{equation}\label{psiStuffEqn}
Y^2\int_{0}^{Y^{-1}} \#\{a\in T\cap R: |a^{-1}|>s \} ds.
\end{equation}
We note that an element $a$ has $|a^{-1}|>s$ if and only if $a$ has absolute value at most $s^{-1}$ at some infinite place. Furthermore, by construction, if $a$ is in $R$, it must have absolute value at most $O(X^{1/d})$ at each real place. Let $M^{\nu}_s$ be the set of $a\in R\cap T$ so that $|a|_\nu\leq s^{-1}$ for some particular infinite place $\nu$. Hence, $\#\{a\in T\cap R: |a^{-1}|>s \}\leq \sum |M^\nu_s|.$ Pick a fundamental domain for $T$. Let $\tilde M^{\nu}_s$ be the union of translates of this fundamental domain by elements of $M^\nu_s$. It is clear that $|M^\nu_s| = O(\textrm{Vol}(\tilde M^\nu_s))$. On the other hand, for $a\in \tilde M^\nu_s$, $|a|_\nu \leq s^{-1}+O(1)$ and $|a|_\mu =O(X^{1/d})$ for other infinite places $\mu$. Thus, $\textrm{Vol}(\tilde M^\nu_s)=O(X^{1-1/d}(1+s^{-1})).$ Hence for any $s$ we have that
$$
\#\{a\in T\cap R: |a^{-1}|>s \} = O(\min(X,X^{1-1/d}s^{-1})).
$$
Thus the quantity in Equation \eqref{psiStuffEqn} is at most
\begin{align*}
Y^2 \int_0^{Y^{-1}} O(\min(X,X^{1-1/d}s^{-1})) ds & = O(Y^2 X^{1-1/d} \log(X^{1/d}/Y))\\ & = O(Y^2 X^{1-1/d}\log(X)).
\end{align*}
Thus the total contribution from these terms to our final sum is
$$
O(YX^{1-1/d}\log(X)),
$$
which is within our desired bounds.
\end{proof}

Applying Abel summation and Lemma \ref{IdealSumLem} yields the following Corollary.

\begin{cor}\label{IdealLogSumCor}
Fix $L$ a number field of degree $d$, and $\xi$ a Grossencharacter of modulus $\mf m$.  Then given a positive number $X$ and a real number $\alpha$ which has a rational approximation of denominator $q$, we have that
$$
\left|\sum_{N(\mf a) \leq X} \log(N(\mf a))\xi(\mf a) e(\alpha N(\mf a)) \right| = O\left(X\log(X)\left(\frac{1}{q} +\frac{1}{X^{1/d}}+ \frac{q}{X} \right)^{10^{-d}/2} \right).
$$
\end{cor}

\subsubsection{Bounds on $F$}\label{FRoughBoundSec}

We are finally ready to prove our bound on $F$.

\begin{prop}\label{FBoundRoughProp}
Fix a number field $L$ of degree $d$ and a Grossencharacter $\xi$.  Let $X\geq 0$ be a real number.  Let $\alpha$ be a real number with a rational approximation of denominator $q$ where $XB^{-1}>q>B$ for some $B>0.$  Then $F_{L,\xi,X}(\alpha)$ is
$$
O\left( X\left(\log^2(X)B^{-10^{-d}/12}+\log^2(X)X^{-10^{-d}/60}+\log^2(X)X^{-10^{-d}/10d}+\log^{2+d^2/2}(X)B^{-1/12} \right) \right).
$$
Where the asymptotic constant may depend on $L$ and $\xi$, but not on $X,q,B$ or $\alpha$.
\end{prop}

Note that a bound for $F_{L,\xi,X}(\alpha)$ is already known for the case when $L/\Q$ is abelian.  In \cite{ExpSumsOverPrimesInArithmeticProgressions}, bounds are established for exponential sums over primes in an arithmetic progression.  By Theorem \ref{CFTTheorem}, this is equivalent to proving bounds on $F$ (or more precisely, $G$) when $L$ is abelian over $\Q$.  Proposition \ref{FBoundRoughProp} can be thought of as a generalization of this result.

\begin{proof}
Our proof is along the same lines as Theorem 13.6 of \cite{IK}.  We first note that the suitable generalization of Equation (13.39) of \cite{IK} still applies.  Letting $y=z=X^{2/5}$  ($y$ and $z$ are variables used in (13.39) of \cite{IK}), we find that $F_{L,\xi,X}(\alpha)$ equals
\begin{align*}
&\sum_{\substack{N(\mf{ab})\leq X\\ N(\mf a)<X^{2/5}}} \mu(\mf a)\xi(\mf a)\log(N(\mf b))\xi(\mf b)e(\alpha N(\mf a)N(\mf b)) \\
&- \sum_{\substack{N(\mf{abc})\leq X\\ N(\mf b),N(\mf c)\leq X^{2/5}}}\mu(\mf b)\Lambda_L(\mf c)\xi(\mf{bc})\xi(\mf a)e(\alpha N(\mf{bc})N(\mf a))\\
&+ \sum_{\substack{N(\mf{abc})\leq X\\ N(\mf b),N(\mf c) \geq X^{2/5}}}\mu(\mf b)\xi(\mf b) \Lambda_L(\mf c)\xi(\mf{ac})e(\alpha N(\mf b)N(\mf{ac})) + O(X^{2/5}).
\end{align*}
We bound the first term by applying Corollary \ref{IdealLogSumCor} to the sum over $\mf b$.  Let $A=B^{1/4}\leq X^{1/8}$. By Lemmas \ref{RatApproxStructureLem} and \ref{NormMultLem}, we can bound the sum over terms where $\alpha N(\mf a)$ has no rational approximation with denominator between $A$ and $\frac{X}{AN(\mf a)}$ by
\begin{align*}
O\left(X\left(\log^2(X)  \left( A^3B^{-1} + X^{-3/5}A^4\right)\left( \frac{B}{A}\right)^{\epsilon} \right)\right)
 = O\left(X\log^2(X)B^{-1/4+\epsilon} \right).
\end{align*}
For other values of $\mf b$, Corollary \ref{IdealLogSumCor} bounds the sum as
$$
O\left(X\log^2(X)\left(B^{-1/4} + X^{-3/5d} \right)^{10^{-d}/2} \right).
$$

The second term is bounded using similar considerations.  We let $A=\min(B^{1/4},X^{1/41})$, and use Lemmas \ref{RatApproxStructureLem} and \ref{NormMultLem} to bound the sum over terms with $\mf b$ and $\mf c$ such that $N(\mf{bc})\alpha$ has no rational approximation with norm between $A$ and $\frac{X}{AN(\mf{bc})}$ by
\begin{align*}
O\left(X\log^3(X)\left(\frac{B}{A}\right)^\epsilon \left( B^{-1/4}+ X^{-1} A^4 X^{4/5}\right) \right)\\ = O\left(X\log^3(X)\left(B^{-1/4+\epsilon}+ X^{-1/10} \right) \right).
\end{align*}
Using Lemma \ref{IdealSumLem}, we bound the sum over other values of $\mf b$ and $\mf c$ as
$$
O\left(X\log^2(X)\left(A^{-1}+X^{-1/5d} \right)^{10^{-d}/2} \right).
$$

To bound the last sum, we first change to a sum over $\mf b$ and $\mf d = \mf a \cdot \mf c$.  We have coefficients
$$
x(\mf b) = \mu(\mf b)\xi(\mf b),
$$
and
$$
y(\mf d) = \sum_{\substack{\mf a \cdot\mf c = \mf d\\ N(\mf c) \geq X^{2/5}}} \Lambda_L(\mf c)\xi(\mf{ac}).
$$
We note that $|y(\mf d)| \leq \log(N(\mf d)) \leq \log(X)$.  Our third term then becomes
$$
\sum_{\substack{N(\mf{bd})\leq X\\ N(\mf b),N(\mf d) \geq X^{2/5}}} x(\mf b)y(\mf d) e(\alpha N(\mf b) N(\mf d)).
$$
We apply the bilinear form method.  First, we split the sum over $\mf b$ into parts based on which dyadic interval (of the form $[K,2K]$), the norm of $\mf b$ lies in.  Next, for each of these summands, we apply Cauchy-Schwartz to bound it by
\begin{align*}
& \left(\Bigg(\sum_{N(\mf b)\in [K,2K]} |x(\mf b)|^2 \Bigg)\cdot\Bigg(\sum_{N(\mf b)\in [K,2K]} \Bigg(\sum_{\substack{N(\mf d) \leq X/N(\mf b)\\N(\mf d)\geq X^{2/5}}}y(\mf d) e(\alpha N(\mf{bd})) \Bigg)^2\Bigg) \right)^{1/2}\\
\ll& K^{1/2}\Bigg(\sum_{\substack{N(\mf b)\in [K,2K]\\N(\mf d),N(\mf d')\leq X/N(\mf b)\\N(\mf d),N(\mf d')\geq X^{2/5}}} y(\mf d) \overline{y(\mf d')} e(\alpha N(\mf b)(N(\mf d) - N(\mf d'))) \Bigg)^{1/2}\\
\ll & K^{1/2}\log(X)\Bigg(\sum_{\substack{X^{2/5} \leq N(\mf d),N(\mf d')\\ N(\mf d),N(\mf d')\leq X/(2K)}}\Bigg|\sum_{\substack{N(\mf b)\in [K,2K]\\N(\mf b)\leq X/N(\mf d)\\N(\mf b)\leq X/N(\mf d')}} e(\alpha (N(\mf d)-N(\mf d'))N(\mf b))\Bigg| \Bigg)^{1/2}.
\end{align*}
We let $A = \min(B^{1/6},X^{1/16})$.  We bound terms separately based on whether or not $\alpha(N(\mf d) - N(\mf d'))$ has a rational approximation with denominator between $A$ and $KA^{-1}$.  Applying Lemma \ref{RatAproxMultLem} with $X=K$, $Y=XK^{-1}$, $A=A$ and $B=B$, we get that the number of values of $N(\mf d) - N(\mf d')$ that cause this to happen is
$$
O\left(XK^{-1}\left(B^{-1/6}+X^{3/32}B^{-1/2}+\log(X)A^4K^{-1} \right) \right)=O(XK^{-1}B^{-1/6}).
$$
For each such difference, $m$, the number of pairs $\mf d,\mf d'$ with norms at most $X/2K$ and $N(\mf d)-N(\mf d')=m$ is
$$
\sum_{n=1}^{X/2K-m} \left(\textrm{Number of ideals with norm } n \right)\cdot \left(\textrm{Number of ideals with norm } n + m \right),
$$
which by Cauchy-Schwartz is at most
$$
\sum_{n=1}^{X/2K} \left(\textrm{Number of ideals with norm } n \right)^2.
$$
Letting $W(n)$ be the number of ideals of $L$ with norm $n$, we have by Lemma 1.1 of \cite{IdealNorm} that $W(n) \leq \tau_d(n)$, where $\tau_{d}(n)$ is the number of ways of writing $n$ as a product of $d$ integers. We therefore have that $W^2(n)\leq \tau_d(n)^2 \leq \tau_{d^2}(n)$ and hence the above sum is $O(XK^{-1}\log^{d^2}(X))$.  Hence the total contribution from terms with such $\mf d$ and $\mf d'$ is at most
\begin{align*}
O\left(\left(K^{1/2}\log(X)\right)\left(K\left(XK^{-1}B^{-1/6}\right)\left(XK^{-1}\log^{d^2}(X)\right) \right)^{1/2}\right)\\
 = O\left(X\log^{1+d^2/2}(X)B^{-1/12} \right).
\end{align*}
The sum over the $O(\log(X))$ possible values for $K$ of the above is
$$
O\left(X\log^{2+d^2/2}(X)B^{-1/12} \right).
$$
On the other hand, the sum over $\mf d$ and $\mf d'$ so that $\alpha(N(\mf d)-N(\mf d'))$ has a rational approximation with appropriate denominator is bounded by Lemma \ref{IdealSumLem} by
\begin{align*}
O\left(\left(K^{1/2}\log(X)\right)\left(\left(XK^{-1} \right)^2K\left(A^{-1}+K^{-1/d} \right)^{10^{-d}/2} \right)^{1/2} \right)\\
= O\left(X\log(X)\left(A^{-1}+K^{-1/d} \right)^{10^{-d}/4} \right).
\end{align*}
Summing over all of the intervals we get
$$
O\left(X\log^2(X)\left(A^{-1}+X^{-2/5d}\right)^{10^{-d}/4} \right).
$$
Putting this all together, we get the desired bound for $F$.
\end{proof}

\subsection{Putting it Together}

We are finally prepared to prove Theorem \ref{FApproxThm}.
\begin{proof}
We note that $\alpha$ can always be approximated by $\frac{a}{q}$ for some relatively prime integers $a,q$ with $q\leq X\log^{-B}(X)$ so that
$$
\left| \alpha - \frac{a}{q}\right| \leq \frac{1}{qX\log^{-B}(X)}.
$$
We split into cases based upon weather $q\leq z = \log^B(X)$.

If $q\leq z$ our result follows from Corollary \ref{FsmoothCor}.

If $q\geq z$, our result follows from Propositions \ref{FsharpRoughProp} and \ref{FBoundRoughProp}.
\end{proof}

\section{Approximation of $G$}\label{GSec}

In this Section, we prove Theorem \ref{GApproxThm}.

\begin{proof}
Recall Proposition \ref{GFRelationProp} which states that
$$
G_{K,C,X}(\alpha) = \frac{|C|}{|G|}\left(\sum_\chi \overline{\chi}(c)F_{L,\chi,X}(\alpha) \right) + O(\sqrt{X}).
$$
Where $c$ is some element of $C$, and $L$ is the fixed field of $\langle c \rangle \subset \Gal(K/\Q)$.
Applying Theorem \ref{FApproxThm}, this is within $O\left(X\log^{-A}(X)\right)$ of
\begin{align*}
\frac{|C|}{|G|}\left(\sum_\chi \overline{\chi}(c)F_{L,\chi,X,z}^\sharp(\alpha) \right)
&= \frac{|C|}{|G|}\left(\sum_\chi \sum_{n\leq X} \overline{\chi}(c)\Lambda_{L/\Q}(n)\Lambda_z(n)\chi(n)e(\alpha n) \right)\\
&= \frac{|C|}{|G|}\left(\sum_{n\leq X} \Lambda_z(n)e(\alpha n) \left(\Lambda_{L/\Q}(n)\sum_\chi\overline{\chi}(c)\chi(n)  \right)\right).
\end{align*}
Note that in the above, $\chi$ is summed over characters of $\Gal(K/L)$ and that $\chi(n)$ is taken to be 0 unless $\chi$ can be extended to a character of $\Gal(K/\Q)^{ab}$.  We wish the evaluate the inner sum over $\chi$ for some $n\in H_L$.

Let the kernel of the map $\langle c\rangle \rightarrow \Gal(K/\Q)^{ab}$ be generated by $c^k$ for some $k|\ord(c)$.  Then $\chi(n)$ is 0 unless $\chi(c^k)=1$.  Therefore we can consider the sum as being over characters $\chi$ of $\langle c\rangle/c^k$.  Taking $K^a$ to be the maximal abelian subextension of $K$ over $\Q$, this sum is then $k$ if $[K^a/\Q,n]=c$ and 0 otherwise.  Hence the sum over $\chi$ is non-zero if and only if $n\in H_C$.  The index of $H_C$ in $H_L$ is $[H_L:H_K]$, which is in turn the size of the image of $\langle c\rangle$ in $\Gal(K/\Q)^{ab}$, or $|\langle c\rangle/\langle c^k \rangle| = k$.  Hence $\Lambda_L(n)\sum_\chi\overline{\chi}(c)\chi(n) = \Lambda_{K,C}(n)$.  Therefore $G_{K,C,X}(\alpha)$ is within $O\left( X\log^{-A}(X)\right)$ of
$$
\frac{|C|}{|G|}\sum_{n\leq X}\Lambda_{K,C}(n)\Lambda_z(n)e(\alpha n) = G_{K,C,X,z}^\sharp(\alpha).
$$
\end{proof}

\section{Proof of Theorem \ref{mainThm}}\label{proofSec}

We now have all the tools necessary to prove Theorem \ref{mainThm}.  Our basic strategy will be as follows.  We first define a generating function $H$ for the number of ways to write $n$ as $\sum_i a_ip_i$ for $p_i$ primes satisfying the appropriate conditions.  It is easy to write $H$ in terms of the function $G$.  First, we will show that if $H$ is replaced by $H^\sharp$ by replacing these $G$'s by $G^\sharp$'s, this will introduce only a small change (in an appropriate norm).  Dealing with $H^\sharp$ will prove noticeably simpler than dealing with $H$ directly.  We will essentially be able to approximate the coefficients of $H^\sharp$ using sieving techniques.  Finally we combine these results to prove the Theorem.

\subsection{Generating Functions}\label{ErrorTermSec}

We begin with some basic definitions.
\begin{defn}
Let $K_i,C_i,a_i,X$ be as in the statement of Theorem \ref{mainThm}.  Then we define
$$
S_{K_i,C_i,a_i,X}(N) := \sum_{\substack{p_i \leq X \\ [K_i/\Q, p_i] = C_i \\ \sum_i a_i p_i = N}} \prod_{i=1}^k \log(p_i).
$$
(i.e. the left hand side of Equation \eqref{asymptEqu}).  We define the generating function
$$
H_{K_i,C_i,a_i,X}(\alpha) := \sum_N S_{K_i,C_i,a_i,X}(N)e(N\alpha ).
$$
Notice that this is everywhere convergent since there are only finitely many non-zero terms.
\end{defn}
We know from basic facts about generating functions that
\begin{equation}\label{genFunctProdEqn}
H_{K_i,C_i,a_i,X}(\alpha) = \prod_{i=1}^k G_{K_i,C_i,X}(a_i\alpha).
\end{equation}
We would like to approximate the $G$'s by corresponding $G^\sharp$'s.  Hence we define
\begin{defn}
$$
H_{K_i,C_i,a_i,z,X}^\sharp(\alpha) := \prod_{i=1}^k G_{K_i,C_i,z,X}^\sharp(a_i\alpha).
$$
$$
H_{K_i,C_i,a_i,z,X}^\flat(\alpha) := H_{K_i,C_i,a_i,X}(\alpha) -H_{K_i,C_i,a_i,z,X}^\sharp(\alpha).
$$
\end{defn}
We now prove that this is a reasonable approximation.

\begin{lem}\label{HapproxLem}
Let $A$ be a constant, and $z=\log^B(X)$ for $B$ a sufficiently large multiple of $A$.
If $k\geq 3$,
$$
\left| H_{K_i,C_i,a_i,z,X}^\flat \right|_1 = O(X^{k-1}\log^{-A}(X)).
$$
If $k=2$,
$$
\left| H_{K_i,C_i,a_i,z,X}^\flat \right|_2 = O(X^{3/2}\log^{-A}(X)).
$$
In the above we are taking the $L^1$ or $L^2$ norm respectively of $H_{K_1,C_i,a_i,X}^\flat$ as a function on $[0,1]$, and the asymptotic constants in the big-$O$ terms are allowed to depend on $K_i$, $a_i$, $A$ and $B$, but not on $X$ or $N$.
\end{lem}
\begin{proof}
Our basic technique is to write each of the $G$'s in Equation \ref{genFunctProdEqn} as $G^\sharp+G^\flat$ and to expand out the resulting product.  We are left with a copy of $H^\sharp$ and a number of terms which are each a product of $k$ $G^\sharp$ or $G^\flat$'s, where each such term has at least one $G^\flat.$  We need several facts about various norms of the $G^\sharp$ and $G^\flat$'s.  We recall that the squared $L^2$ norm of a generating function is the sum of the squares of it's coefficients.

\begin{itemize}
\item By Theorem \ref{GApproxThm}, the $L^{\infty}$-norm of $G^\flat$ is $O\left( X \log^{-2A-k}(X)\right)$.
\item The $L^{\infty}$ norm of $G^\sharp$ is clearly $O\left( X\log\log(X)\right)$.
\item $|G^\sharp|_2^2 = O(X \log\log^2(X))$.
\item $|G|_2^2 = O(X\log(X))$.
\item Combining the last two statements, we find that $|G^\flat|_2^2 = O(X\log(X))$.
\end{itemize}

For $k\geq 3$, we note that by Cauchy-Schwartz, the $L^1$ norm of a product of $k$ functions is at most the products of the $L^2$ norms of two of them times the products of the $L^\infty$ norms of the rest.  Using this and ensuring that at least one of the terms we take the $L^\infty$ norm of is a $G^\flat$, we obtain our bound on $|H^\flat|_1$.

For $k=2$, we note that the $L^2$ norm of a product of two functions is at most the $L^2$ norm of one times the $L^\infty$ norm of the other.  Applying this to our product, ensuring that we take the $L^\infty$ norm of a $G^\flat$ we get the desired bound on $|H^\flat|_2$.
\end{proof}

\subsection{Dealing with $H^\sharp$}\label{mainTermSec}

Now that we have shown that $H^\sharp$ approximates $H$, it will be enough to compute the coefficients of $H^\sharp$.

\begin{prop}\label{HSharpProp}
Let$z=\log^B(X)$ for $B$ some positive constant. Pick $\epsilon>0$ some other constant
The $e(N\alpha)$ coefficient of $H_{K_i,C_i,a_i,z,X}^\sharp(\alpha)$ is given by the right hand side of Equation \eqref{asymptEqu}, or
$$
\left(\prod_{i=1}^k \frac{|C_i|}{|G_i|} \right)C_\infty C_{D} \left(\prod_{p \nmid D} C_p\right) + O\left(X^{k-1+\epsilon - 1/(3B)} \right),
$$
where the implied constant above depends potentially on $k$, $K_i$, $a_i$, $B$ and $\epsilon$, but not on $X$ or $N$.
\end{prop}
\begin{proof}
We note that the quantity of interest is equal to
\begin{equation}\label{HSharpCoefEqn}
\left( \prod_{i=1}^k \frac{|C_i|}{|G_i|}\right)\sum_{\substack{n_1,\ldots,n_k\leq X\\ \sum_{i=1}^k a_i n_i = N}} \left(\prod_{i=1}^k \Lambda_{K_i,C_i}(n_i) \right)\left(\prod_{i=1}^k \Lambda_z(n_i) \right).
\end{equation}
This is
\begin{align*}
& \left( \prod_{i=1}^k \frac{|C_i|}{|G_i|}\right)C(z)^k \left(\prod_{i=1}^k\frac{\phi(D)}{|H_i|}  \right) \\ & \cdot \left|\left\{(n_1,\ldots,n_k)\in \{1,2,\ldots,X\}^k:n_i \pmod D \in H_i, (n_i,P(z))=1, \sum_{i=1}^k a_i n_i = N \right\} \right|.
\end{align*}
Thus our problem reduces to computing the size of the set $S$ given by:
$$
\left\{(n_1,\ldots,n_k)\in \{1,2,\ldots,X\}^k:n_i \pmod D \in H_i, (n_i,P(z))=1, \sum_{i=1}^k a_i n_i = N \right\}.
$$

Our main technique for dealing with this term will be based of sieving. In particular, we sieve based on which primes less than $z$ divide any of the $n_i$. For $d|P(z,D)$ we define
$$
S_d = \left\{(n_1,\ldots,n_k)\in \{1,2,\ldots,X\}^k:n_i \pmod D \in H_i, d\bigg|\prod_{i=1}^k n_i, \sum_{i=1}^k a_i n_i = N \right\}.
$$
It follows easily that
$$
|S| = \sum_{d|P(z,D)} \mu(d)|S_d|.
$$
Thus, it suffices to estimate the sizes of the $S_d$.

We note that $S_d$ is the set of tuples $(n_1,\ldots,n_k)$ with $n_i\leq X$, and $\sum_{i=1}^k a_i n_i = N$ so that the vector $(n_1,\ldots,n_k)\pmod{dD}$ lies in some restricted set of congruence classes. In particular, let
$$
T_D:= \left\{(n_1,\ldots,n_k)\in (\Z/D\Z)^k: \sum_{i=1}^k a_i n_i \equiv N \pmod{D}, n_i\in H_i\right\},
$$
and
$$
T_p:= \left\{(n_1,\ldots,n_k)\in (\Z/D\Z)^k: \sum_{i=1}^k a_i n_i \equiv N \pmod{p}, p\bigg| \prod_{i=1}^k n_i\right\}.
$$
Then the elements of $S_d$ are the tuples with $n_i\leq X$, $\sum_{i=1}^k a_i n_i = N$, and $(n_1,\ldots,n_k)\in T_D$ and $(n_1,\ldots,n_k)\in T_p$ for all $p|d$. To count the number of such points, we first condition on their congruence classes modulo $dD$. In particular, by the Chinese Remainder Theorem, a $(n_1,\ldots,n_k)\in S_d$ can take on only $|T_D| \prod_{p|d} |T_p|$ different possible congruence classes modulo $dD$. Fixing such a class, $c\in (\Z/Dd\Z)^k$ with $c\pmod D \in T_D$ and $c\pmod p \in T_p$ for $p|d$, the set of elements of $S_d$ congruent to $c$ are simply those tuples with $n_i\leq X$, $\sum_{i=1}^k a_i n_i=N$ and $(n_1,\ldots,n_k)\equiv c \pmod{dD}.$ Therefore, we have that $|S_d|$ is the sum over such $c$ of
\begin{equation}\label{LcSetEqn}
\left| \left\{(n_1,\ldots,n_k) \in \{1,\ldots,X\}^k: \sum_{i=1}^k a_i n_i = N, (n_1,\ldots,n_k) \equiv c \pmod{dD} \right\}\right|.
\end{equation}

Notice that the set of $k$-tuples of integers $n_i$ with $\sum_{i=1}^k a_i n_i = N$ and $(n_1,\ldots,n_k)\equiv c \pmod{dD}$ is an affine lattice within the space $V$ of tuples of real numbers $x_i$ so that $\sum_{i=1}^k a_i x_i = N$. We induce a measure on $V$ from the standard measure on $\R^n$ by putting the measure ${\sum_{i=1}^k a_i dx_i}$ on the quotient. We note that under this measure, $C_\infty$ is the measure of $R:=[0,X]^k \cap V$. The lattice $\Z^k$ has covolume $1$ in $\R^k$. Since the $a_i$ are relatively prime, the image of $\Z^k$ under $(n_1,\ldots,n_k) \rightarrow \sum_{i=1}^k a_i n_i$ is $\Z$. Thus the projection of $\Z^k$ to $\R^k/V$ has covolume $1$. Therefore $\Z^k\cap V$ has covolume $1$ within $V$. Let $L$ be the affine lattice $\Z^k \cap V$. For $c\in T_d \times \prod_{p|d} T_p$, let $L_c$ be the sublattice of $L$ consisting of elements congruent to $c$ modulo $dD$. The covolume of $L_c$ is $(dD)^{k-1}$ times the covolume of $L$, and is thus $(dD)^{k-1}$. Notice that the set in Equation \eqref{LcSetEqn} is exactly $L_c\cap R$. We now try to estimate its size.

Consider a fundamental domain $M$ of $L_c$. We can construct $M$ so that it has diameter $O(d)$. Take the union of translates of $M$ centered at the elements of $L_c\cap R$. It is clear that this produces a set whose symmetric difference with $R$ is contained within the set of points within distance $O(d)$ of the boundary of $R$. It is thus, easy to see that this union has volume $\textrm{Vol}(R) + O(dX^{k-2}+d^{k-1})$. On the other hand the volume of this region equals the covolume of $L_c$ times the number of points in $L_c\cap R$. Thus,
$$
|L_c\cap R| = \frac{C_\infty + O(dX^{k-2}+d^{k-1})}{(dD)^{k-1}}.
$$
Therefore, $|S_d|$ is the sum over $T_D\times \prod_{p|d} T_p$ of this quantity, or
$$
|S_d| = |T_D|\cdot \prod_{p|d} |T_p| (C_\infty + O(d X^{k-2})) (dD)^{1-k}.
$$

In order to obtain proper control on the error term above, we will want to bound the size of $T_p$. For a tuple $(n_1,\ldots,n_k)\in T_p$ at least one of the $n_j$ must be zero modulo $p$. Fixing such an $j$, it must still be the case that $\sum_{i=1}^k a_i n_i \equiv N \pmod{p}$. Unless $a_i=0$ for all $i\neq j$ and $N\equiv 0\pmod{p}$, there are only $p^{k-2}$ such solutions. If $p|N$ and $p|a_i$ for all $i\neq j$, we claim that our proposition holds trivially. In particular, we have that $a_j$ is not divisible by $p$ (since the $a_i$ are relatively prime). Therefore if we have integers $n_i$ with $\sum_{i=1}^k a_i n_i \equiv N \pmod{p}$, then $n_j$ must be divisible by $p$. Therefore $S$ is empty, and $C_p$ is 0, and so both of sides of the equation in question are 0. Hence, we may assume that this is not the case and therefore assume that $|T_p|\leq kp^{k-2}$ for all $p$. Thus the error term above is at most
$$
O\left( \prod_{p|d} kp^{k-2} (d^{2-k}X^{k-2} + 1) \right) = O\left(d^\epsilon (X+d)^{k-2} \right).
$$

While the above bound will prove sufficient for $d\ll X$, we will need a different bound for larger values of $d$. We claim for any $d$ that $|S_d|=O(X^{k-1}d^{\epsilon-1})$. This is because for $(n_1,\ldots,n_k)\in S_d$, we must have some $d_i$ with $d_i|n_i$ and $d=\prod_{i=1}^k d.$ There are $\tau_k(d)=d^{\epsilon}$ ways to pick the $d_i$. For each way of picking the $d_i$, the set of points in $L$ with $d_i|n_i$ for each $i$ forms a lattice of covolume $d$. If any $d_i$ is bigger than $X$, there is nothing to prove. Otherwise, this lattice has a fundamental domain of diameter $O(X)$ and thus extending translates of this fundamental domain around each point in the intersection of this lattice with $R$ yields a figure of volume $O(X^{k-2})$. Thus, the number of such points is $O(X^{k-2}d^{-1})$. Thus, $|S_d|=O(X^{k-2}d^{\epsilon-1})$.

To summarize, we have that for $d\leq X$, we have that
$$
|S_d| = |T_D|\cdot \prod_{p|d} |T_p| C_\infty (dD)^{1-k} + O(X^{k-2+\epsilon}).
$$
And for $d\geq X$, we have that
$$
|S_d| = O(X^{k-1}d^{\epsilon-1}) = |T_D|\cdot \prod_{p|d} |T_p| C_\infty (dD)^{1-k} + O(X^{k-1}d^{\epsilon-1}) .
$$
Thus,
\begin{align*}
|S| = & |T_D| D^{1-k} C_\infty \sum_{d|P(z,D)} \prod_{p|d} \frac{-|T_p|}{p^{k-1}}\\ & + \sum_{d|P(z,D),d\leq X} O(X^{k-2+\epsilon}) + \sum_{d|P(z,D),d\geq X} O(X^{k-1}d^{\epsilon-1}).
\end{align*}
We begin by dealing with the error term above. By Corollary \ref{SmoothNumbersCor}, it is at most
\begin{align*}
& O(X^{k-2+\epsilon}S(z,X)) + \int_{X}^\infty O(X^{k-1})S(z,y)y^{\epsilon-2}dy\\
= & O(X^{k-1+\epsilon-1/(3B)}) + \int_X^\infty O(X^{k-1}) y^{\epsilon-1/(3B)-1}dy\\
= & O(X^{k-1+\epsilon-1/(3B)}).
\end{align*}

Therefore, this term can be safely ignored, and up to acceptable error we may approximate $|S|$ as
\begin{align*}
|T_D| D^{1-k} C_\infty \sum_{d|P(z,D)} \prod_{p|d} \frac{-|T_p|}{p^{k-1}} = \frac{|T_D|}{D^{k-1}} C_\infty \prod_{p|P(z,D)} \left(1-\frac{|T_p|}{p^{k-1}} \right).
\end{align*}
Thus, up to acceptable error, the coefficient in question is equal to
\begin{align*}
&\left(\prod_{i=1}^k \frac{|C_i|}{|G_i|} \right)C_\infty \left( \frac{|T_D|}{D^{k-1}} \right) \left(\prod_{p|D} (1-p^{-1})^{-k} \right)\left(\prod_{p|P(z,D)} \left(1-\frac{|T_p|}{p^{k-1}} \right) (1-p^{-1})^{-k} \right)\\
& = \left(\prod_{i=1}^k \frac{|C_i|}{|G_i|} \right)C_\infty \left( \frac{D|T_D|}{\phi(D)^{k}} \right)\left(\prod_{p|P(z,D)} \frac{p(p^{k-1}-|T_p|)}{(p-1)^k}\right) \\
& = \left(\prod_{i=1}^k \frac{|C_i|}{|G_i|} \right)C_\infty C_D \left(\prod_{p|P(z,D)} C_p\right).
\end{align*}
This completes our proof.
\end{proof}

\subsection{Putting it Together}

We are finally able to prove Theorem \ref{mainThm}
\begin{proof}
Let $B$ be a sufficiently large multiple of $A$, and $z=\log^B(X)$.

For $k\geq 3$ we have that
$$
S_{K_i,C_i,a_i,X}(N) = \int_0^1 H_{K_i,C_i,a_i,X}(\alpha)e(-N\alpha).
$$
By Lemma \ref{HapproxLem} this is
$$
\int_0^1 H_{K_i,C_i,a_i,z,X}^\sharp(\alpha)e(-N\alpha)
$$
up to acceptable errors.  This is the $e(N\alpha)$ coefficient of $H_{K_i,C_i,a_i,z,X}^\sharp(\alpha)$, which by Proposition \ref{HSharpProp} is as desired.

For $k=2$, we let $T_{K_i,C_i,a_i,X}(N)$ be the corresponding right hand side of Equation \ref{asymptEqu}.  It will suffice to show that
$$
\sum_{|n|\leq \sum_i |a_i|X} (S_{K_i,C_i,a_i,X}(N)-T_{K_i,C_i,a_i,X}(N))^2 = O(X^3\log^{-2A}(X)).
$$
If we define the generating function
$$
J_{K_i,C_i,a_i,X}(\alpha) = \sum_{|N|\leq \sum_i |a_i|X} T_{K_i,C_i,a_i,X}(N)e(N\alpha)
$$
we note that the above is equivalent to showing that
$$
|H_{K_i,C_i,a_i,X}-J_{K_i,C_i,a_i,X}|_2 = O(X^{3/2}\log^{-A}(X)).
$$
But by Lemma \ref{HapproxLem}, we have that
$$
|H_{K_i,C_i,a_i,X}-H_{K_i,C_i,a_i,z,X}^\sharp|_2 = O(X^{3/2}\log^{-A}(X)),
$$
and by Proposition \ref{HSharpProp}, we have
$$
|H_{K_i,C_i,a_i,z,X}^\sharp-J_{K_i,C_i,a_i,X}|_2 = O(X^{3/2}\log^{-A}(X)).
$$
This completes the proof.
\end{proof}

\section{Application}\label{appSec}

We present an application of Theorem \ref{mainThm} to the construction of elliptic curves whose discriminants are divisible only by primes with certain splitting properties.

\begin{thm}
Let $K$ be a number field.  Then there exists an elliptic curve defined over $\Q$ so that all primes dividing its discriminant split completely over $K$.
\end{thm}
\begin{proof}
We begin by assuming that $K$ is a normal extension of $\Q$.  We will choose an elliptic curve of the form:
$$
y^2 = X^3 + AX + B.
$$
Here we will let $A=pq/4$, $B=npq^2$ where $n$ is a small integer and $p,q$ are primes that split over $K$.  The discriminant is then
\begin{align*}
-16(4A^3+27B^3) & = -64p^3q^3/64 - 432n^2p^2q^4\\
& = -p^2 q^3 (p + 432 n^2 q).
\end{align*}
Hence it suffices to find primes $p,q,r$ that split completely over $K$ with $ p + 432n^2 q - r = 0$.  We do this by applying Theorem \ref{mainThm} with $k=3$, $K_i=K$, $C_i = \{e\}$, and $X$ large.  As long as $C_D>0$ and $C_p>0$ for all $p$, the main term will dominate the error and we will be guaranteed solutions for sufficiently large $X$.   If $n=D$, this will hold.  This is because for $C_D$ to be non-zero we need to have solutions $n_1 + 0 n_2 - n_3 \equiv 0 \pmod D$ with $n_i$ all in some particular subgroup of $(\Z/D\Z)^*$.  This can clearly be satisfied by $n_1=n_3$.  For $p=2$, $C_p$ is non-zero since there is a solution to $n_1 + 0n_2 - n_3 \equiv 0\pmod 2$ with none of the $n_i$ divisible by 2 (take $(1,1,1)$).  For $p>2$, we need to show that there are solutions to $n_1 + 432D^2 n_2 - n_3 \equiv 0\pmod p$ with none of the $n_i$ 0 modulo $p$.  This can be done because after picking $n_2$, any number can be written as a difference of non-multiples of $p$.
\end{proof}

\section{Acknowledgements}

This work was done with the support of an NDSEG graduate fellowship.

\end{document}